\documentclass{amsart}
\usepackage{amsmath,amsthm,amssymb,amsfonts,hyperref}
\usepackage{pdfsync}
\title{The asymptotic growth of torsion homology for arithmetic groups}

\author{Nicolas Bergeron} 

\address{Institut de Math\'ematiques de Jussieu \\
Unit\'e Mixte de Recherche 7586 du CNRS \\
Universit\'e Pierre et Marie Curie \\
4, place Jussieu 75252 Paris Cedex 05, France \\}
\email{bergeron@math.jussieu.fr}
\urladdr{http://people.math.jussieu.fr/~bergeron}

\author{Akshay Venkatesh}
\thanks{We gratefully acknowldge funding agencies that have supported our work. A.V. was partially supported by the Sloan foundation, the Packard Foundation and by an NSF grant.}
\address{Department of Mathematics, Stanford University, Stanford CA 94304}
\email{akshay@math.stanford.edu}

 \DeclareFontFamily{OT1}{rsfs}{}
 \newcommand{\tors}{\mathrm{tors}}
\DeclareFontShape{OT1}{rsfs}{n}{it}{<-> rsfs10}{}
\DeclareMathAlphabet{\mathscr}{OT1}{rsfs}{n}{it}
\newcommand{\Cinf}{V_0}
\newcommand{\detp}{\det{}'}
\newcommand{\prodalt}{ \prod \- {}^*}
\newcommand{\atr}{\mathfrak{a}_{0\R}}
\newcommand{\Bbull}{B^{\ast}}
\newcommand{\image}{\mathrm{image}}
\newcommand{\free}{\mathrm{free}}
\newcommand{\DkE}{\bigwedge^k \mathfrak{p}^* \otimes E}

\newcommand{\mass}{\mathrm{mass}}

\newcommand{\Tr}{\mathrm{tr}}

\newcommand{\Qvbar}{\overline{\Q_v}}
\newcommand{\bx}{\mathbf{x}}
\newcommand{\tx}{\tilde{x}}

\newcommand{\Abull}{A^{\ast}}
\newcommand{\Qbar}{\bar{\Q}}

\newcommand{\Spec}{\mathrm{Spec}}

\newcommand{\LG}{{}^{L}G}
\newcommand{\C}{\mathbb{C}}

\newcommand{\Z}{\mathbb{Z}}

\newcommand{\bt}{\mathfrak{b}}
\newcommand{\muP}{\mu_{\mathrm{Planch}}}
\newcommand{\at}{\mathfrak{a}_0}

\newcommand{\T}{\mathbf{T}}
\newcommand{\Ad}{\mathrm{Ad}}
\newcommand{\sign}{\mathrm{sign}}
\newcommand{\Qpbar}{\overline{\mathbb{Q}_p}}
\newcommand{\adele}{\mathbb{A}}

\newcommand{\trace}{\mathrm{tr}{\ }}

\newcommand{\m}{\mathfrak{m}}

\renewcommand{\a}{\mathfrak{a}}
\newcommand{\p}{\mathfrak{p}}

\newcommand{\gcompact}{\mathfrak{u}}

\newcommand{\psix}{\psi_{\mathbf{x}}}
\newcommand{\Gal}{\mathrm{Gal}}

\renewcommand{\H}{\mathbf{H}}

\newcommand{\hatT}{T^{\vee}}
\newcommand{\hatB}{B^{\vee}}

\DeclareFontFamily{OT1}{rsfs}{}

\DeclareFontShape{OT1}{rsfs}{n}{it}{<-> rsfs10}{}
\DeclareMathAlphabet{\mathscr}{OT1}{rsfs}{n}{it}
\newcommand{\Q}{\mathbb{Q}}
\newcommand{\Hom}{\mathrm{Hom}}

\newcommand{\G}{\mathbf{G}}
\renewcommand{\k}{\mathfrak{k}}

\newcommand{\R}{\mathbb{R}}
\newcommand{\SO}{\mathrm{SO}}

\newcommand{\Aut}{\mathrm{Aut}}
\newcommand{\rank}{\mathrm{rank}}
 
 \newcommand{\test}{\mathcal{Q}}
\swapnumbers
\newtheorem{thm}[subsection]{Theorem}  
\newtheorem{lem}[subsection]{Lemma}         
\newtheorem*{lem*}{Lemma}         
\newtheorem{prop}[subsection]{Proposition}
\newtheorem*{prop*}{Proposition}

\newtheorem{conj}[subsection]{Conjecture}
 
\newtheorem*{lema}{Lemma}
\newtheorem{cor}[subsection]{Corollary}

\theoremstyle{definition}

\numberwithin{equation}{subsection}

\renewcommand{\H}{\mathbb H}  %

\newcommand{\cal}{\mathcal}
\newcommand{\GL}{\mathrm{GL}}
\newcommand{\SL}{\mathrm{SL}}

\newcommand{\End}{\mathrm{End}}
\newcommand{\SU}{\mathrm{SU}}

\newcommand{\PGL}{\mathrm{PGL}}

\newcommand{\vol}{\mathrm{vol}}

\begin{document}

\begin{abstract}  
When does the amount of torsion in the homology of an arithmetic group grow exponentially with the covolume?  We give many examples where this is so, and conjecture precise conditions. 
\end{abstract}
\maketitle
\tableofcontents

\section{Asymptotic torsion}

Let $\Gamma$ be a lattice in a semisimple Lie group $G$. The asymptotic behavior of the Betti numbers $ \dim H_j (\Gamma, \C)$, when $\Gamma$ varies, has been 
studied by a number of authors.   In particular, if $\Gamma_N \leqslant \Gamma$ is
a decreasing sequence of subgroups, with trivial intersection,  the quotient:
$$\frac{ \dim H_j (\Gamma_N , \C) } {[\Gamma: \Gamma_N]}$$
is known \cite{Luck2} to converge to the $j$th $L^2$-Betti number of $\Gamma$; in particular, 
this is nonzero only when the rank\footnote{Here the rank is the {\it complex} rank, i.e. the maximal dimension of a Cartan algebra, split or otherwise.}  of $G$ equals the rank of a maximal compact subgroup $K \subset G$. 
For example, this is the case when $G = \SL_2(\R)$, and $\Gamma$ is a Fuchsian groups. 

The purpose of this paper is to investigate the corresponding question when Betti numbers  
are replaced by the logarithm of the size of the torsion components of homology groups.    The motivation
comes from number theory; see \S \ref{mot-ash}. 

\subsection{}  Locally symmetric spaces associated to $\SL_2 (\C)$ are
hyperbolic $3$-manifolds.  Numerical experiments  (\cite{Grunewald}, and unpublished data computed by Calegari-Dunfield in connection with \cite{CalegariDunfield}) 
as  well as the work of Calegari-Mazur \cite{CalegariMazur} suggest that arithmetic hyperbolic $3$-manifolds should have a lot of torsion in their
homology. Here are two such results:

  \label{exa}
1. (Taylor, \cite[Theorem 4.2]{TaylorR}).  Let $n_1$ and $n_2$ be two nonnegative integers. Consider 
$S_{n_1 , n_2} = {\rm Sym}^{n_1} (\Z [i]^2) \otimes_{\Z} \overline{{\rm Sym}^{n_2} (\Z [i]^2)}$ as a 
$\SL_2 (\Z [i] )$-module. If $n_1 \neq n_2$,  there exists a finite index
(congruence) subgroup $\Gamma \subset \SL_2 (\Z [i])$ such that the homology groups
$$H_1 ( \Gamma ; S_{n_1 , n_2})$$
have a non-trivial torsion subgroup.

2. (Follows from Silver and Williams \cite{SilverWilliams}.) Let $k \subset {\Bbb S}^3$
be a hyperbolic knot whose Alexander polynomial has nontrivial Mahler measure 
(e.g. the figure eight knot). 
Denote by $M_N$ the $N$-th cyclic cover of ${\Bbb S}^3$ branched over $k$.  We 
decompose $H_1 (M_N ; \Z)$ as the direct sum of a free abelian group and a torsion subgroup
$H_1 (M_N)_{\tors}$. 

For sufficiently large $N$, the $3$-manifold $M_N$ is hyperbolic; moreover, $$\lim_{N \rightarrow +\infty} \frac{\log | H_1 (M_{N})_{\tors}|}{N} >0 .$$

\medskip

\subsection{}
It is convenient to define, for a semisimple Lie group $G$, the ``deficiency'' $\delta(G)$ to be the difference $\mathrm{rank}(G) - \mathrm{rank}(K)$.  The quantity $\delta$ is sometimes
called the fundamental rank of $G$. It equals zero
 if and only if $G$ has discrete series, or, equivalently, if and only if $G$ has a compact Cartan subgroup. 
If $S$ is the global symmetric space $G/K$, we sometimes write
$\delta(S)$ for $\delta(G)$. 

\medskip

\noindent
{\em Examples.} $\delta = 0$ for the groups $\SL_2 (\R)$, $\mathrm{SU}_{n,m}, \SO_{n,m} \ \mbox{($nm$ even)}$; $\delta  =1$ for the groups $\SL_2 (\C )$, $\SL_3(\R),  \SO_{n,m} \ \mbox{$nm$ odd}$ (this is a complete list of almost simple groups with $\delta =1$, up to isogeny); 
$\delta = 2$ for the groups $\SL_5(\R), E_6^{\mathrm{split}}$.

\medskip

Now assume that $\mathbf{G}$ is ${\Bbb Q}$-semisimple and $\Gamma \subset \mathbf{G}({\Bbb Q})$ is a congruence lattice.  Put $G = \mathbf{G}(\R)$.  We fix an ``arithmetic'' $\Gamma$-module $M$;
by this we mean that $M$ is a finite rank free $\Z$-module, and there exists an algebraic representation of $\mathbf{G}$
on $M \otimes \Q$ so that $\Gamma$ preserves $M$. 
We assume $\mathbf{G}$ is anisotropic over ${\Bbb Q}$ (equivalently $\Gamma$ is cocompact in $G$), and consider 
  a decreasing sequence of congruence subgroups
  $\Gamma_N \subset \Gamma_{N-1} \subset \dots \subset \Gamma$ 
with the property that $\cap_N \Gamma_N = \{1\}$. 

\begin{conj}  \label{conjmain}   The limit
$$\lim_{N} \frac{\log |H_j (\Gamma_N , M )_{\tors}|  }{[\Gamma: \Gamma_N]}$$
exists for each $j$ and is zero {\em unless} $\delta(S) = 1$ and $j = \frac{\dim(S) - 1}{2}$. 
In that case, it is always positive and equal to a positive constant $c_{G,M}$ (explicitly described in \S \ref{sl3})
times the volume of $\Gamma \backslash S$. 
\end{conj}

This conjecture -- which remains somewhat speculative at this stage -- can be considered
as predicting three different types of behavior: 
\begin{enumerate}
\item If $\delta  = 0$, then there is little torsion whereas $H_j(\Gamma_N, M \otimes \Q)$ is large;
the torsion is almost entirely ``absorbed'' by the characteristic zero homology;
\item If $\delta =1 $, then there is ``a lot'' of torsion but $H_j(\Gamma_N, M \otimes \Q)$ is small.
\item If $\delta \geq 2$, there is ``relatively little''\footnote{We mean this in the weakest possible sense: there is not exponential growth of torsion. This is not to suggest there is {\em no} torsion, nor that the torsion that exists is uninteresting; quite the contrary!}
torsion or characteristic zero homology. 
\end{enumerate}

As for the degree $\frac{\dim(S)-1}{2}$, this is the ``middle dimension''
for torsion classes: there is a duality between dimensions $j$ and $\dim(S) - 1 - j$. 
Compare the situation for the $L^2$-Betti
numbers of $\Gamma$: they vanish unless $\delta = 0$, and even then are nonvanishing
only in degree $\frac{\dim(S)}{2}$ (again the ``middle dimension''). 

We find it very likely that the restriction to {\em congruence} $\Gamma_N$ is essential. That the intersection of the $\Gamma_N$ is trivial is also likely essential. See
\S \ref{pullbacks} for some counterexamples in this direction.

In the present paper, 
we shall show a result in support of the ``large torsion'' direction. Notations as previous, say that $M$ is {\em strongly acyclic} if the spectra of the (form) Laplacian on $M \otimes \C$-valued $i$-forms on $\Gamma_N \backslash S$ are uniformly bounded away from $0$, for all $N, i$. 
Notice that this implies that $H_j(\Gamma_N, M )$ is torsion.

\begin{thm}  \label{main} Notations as previous, suppose that $\delta(S) = 1$. 
Then strongly acyclic arithmetic $\Gamma$-modules $M$ always exist; moreover, for any such, 
\begin{equation} \label{moa} \liminf_{N} \sum_{j }\frac{\log |H_j (\Gamma_N, M)_{\tors}| }{[\Gamma: \Gamma_N]} \geq c_{G ,M} {\rm vol} (\Gamma \backslash S) > 0. \end{equation} 
Here the sum is taken over integers $j$ with the same parity as $\frac{\dim S -1}{2}$. 
 \end{thm}

 \noindent
 {\em Remark.} This theorem is deduced from  Theorem \ref{approxthm}, which remains valid for any sequence
 of cocompact lattices so that the injectivity radius of $\Gamma_N \backslash S$ goes to infinity -- i.e., $\Gamma_N$ need not be subgroups of a fixed lattice $\Gamma$.  
It is possible to formulate a more general version of Theorem \ref{main} in that setting;
however, this  requires ``compatible''
 specifications of modules for each $\Gamma_N$. 

\medskip

\noindent
{\em Examples.} The $\SL_2 (\Z[i])$-module $S_{n_1 , n_2}$ ($n_1 , n_2 \in {\Bbb N}$)
considered in the first example of \S \ref{exa} is strongly acyclic precisely when $n_1 \neq n_2$.
(Our theorem does not apply to this case, since the lattice is not cocompact;   a twisted
variant where it applies is detailed in \S \ref{adjoint}.)
It is likely that using
some of the ideas of \cite{CV} one could obtain corresponding results for  certain
sequences of subgroups of $\SL_2(\Z[i])$, but we do not attempt to do so. 
\medskip

In Theorem \ref{main}, we cannot in general isolate the degree which produces torsion {\em except}
in certain low degree examples.   
\medskip

\noindent
{\em  If $G = \SL_3(\R)$ then
$\liminf_N \frac{\log |H_2 (\Gamma_N, M)_{\tors}| }{[\Gamma: \Gamma_N]} > 0$; \\
if $G = \SL_2(\C)$ then $\lim_N \frac{\log |H_1 (\Gamma_N, M)_{\tors}| }{[\Gamma: \Gamma_N]} = c_{G,M} \vol(\Gamma \backslash S)$.}

\medskip
 Indeed, these refinements result from the following two observations: 
\begin{itemize}
\item[--]
The proof of the Theorem establishes more than \eqref{moa}; it will in fact show that
\begin{equation} \label{mainresult-sharp} \lim_N \sum_j (-1)^{j + \frac{\dim(S)-1}{2}}
\frac{\log |H_j (\Gamma_N, M)_{\tors}| }{[\Gamma: \Gamma_N]} =  c_{G ,M} {\rm vol} (\Gamma \backslash S).\end{equation}
\item[--] On can bound the torsion in $H_0 (M)$ by a polynomial in the index $[\Gamma : \Gamma_N]$, see \S \ref{H0bound}.
\item[--]
One can bound  the torsion in $H_{\dim(S)-1}(M)$, via Poincar\'e duality and the long exact sequence in homology associated to
\begin{equation} \label{torso} M \rightarrow M \rightarrow M/pM.\end{equation}
\end{itemize}

If one assumes the truth of the congruence subgroup property (CSP)
for a cocompact lattice $\Gamma \leqslant \SL_3(\R)$, we obtain in that case also 
$\lim_N \frac{\log |H_2 (\Gamma_N, M)_{\tors}| }{[\Gamma: \Gamma_N]} = c_{G,M} \vol(\Gamma \backslash S)$;
in fact, the CSP and the homology exact sequence for \eqref{torso}
would imply that $H_1(\Gamma, M)$ is small.

\subsection{}

The main tool used in the proof of Theorem \ref{main} is {\em not due to us; it is a remarkable result \cite{Muller2} of W. M{\"u}ller, a generalization
of the ``Cheeger-M{\"u}ller theorem'' \cite{Cheeger, Muller1}}; this result states, loosely speaking, that the
size of torsion groups can be computed by analytic methods.   Beyond this result, the other techniques are also not original and have been used in other contexts  (see \cite{Lott} for instance). 
 These methods
have a combinatorial counterpart, which we briefly discuss in \S \ref{cp}; this generalizes
the work of Silver and Williams, mentioned above.

 We remark that by combinatorial methods it is {\em in principle} possible to verify the conclusion of Theorem \ref{main} for any specific $\Gamma, M$ by direct computation, without appealing to M{\"u}ller's theorem (cf. \S \ref{as-co-to}, Remark).  However, such an analysis would be very complicated, and we have not attempted to carry it out in any case. 

Recently, W. M{\"u}ller has announced a result of similar nature to Theorem \ref{main}
but in a different aspect. He studies the homological torsion for a fixed cocompact lattice in $\SL_2(\C)$
{\em as the module $M$ varies}. The preprint is now online \cite{MuellerArxiv}. 

\subsection{} \label{mot-ash}
Our motivation for studying these questions is arithmetic. 

In \S \ref{Ash} we recall conjectures of A. Ash \cite{Ash, ADP} and others \cite{TaylorR, Fig, Grunewald} that torsion in the homology of arithmetic groups has arithmetic significance: {\em very roughly}, 
a mod $p$ torsion class in the homology of an arithmetic group parameterizes a field extension 
$K/\mathbb{Q}$ whose Galois group is a simple group of Lie type over $\mathbb{F}_p$.
The quantity $\delta$ specifies  the isomorphism class of $K \otimes \R$.

 We make an attempt in \S \ref{bharg} to verify whether Conjecture \ref{conjmain}
 and Theorem \ref{main}
are compatible with Ash's conjectures: are there indeed 
 more such field extensions $K/\mathbb{Q}$ when $\delta=1$?
Our analysis is based on heuristics 
  proposed by Bhargava \cite{Bhargheuristic} and is therefore speculative; nonetheless,
  it seems to us worthy of inclusion. 
Precisely, we show that Ash's conjectures and Bhargava's heuristics imply that the ``likelihood'' of the existence of a mod $p$ class in the homology 
of $\Gamma \leqslant G$ should be of size $\sim p^{-\delta(G)}$. If $\delta(G) = 0$, this amounts simply to the fact
that there is abundant characteristic zero homology, which reduces to give abundant mod $p$ homology. On the other hand, since 
$$\sum_{p} p^{-m} = \begin{cases} \infty, \ \ \ \mbox{ if } m =1 \\ 
\mbox{finite}, \ \ \ \mbox{ if } m > 1, \end{cases} $$ 
this suggests an abundance of torsion precisely when $\delta(G) = 1$.

\subsection{Organization of the paper.}

In the body of the paper, we deal with cohomology rather than homology, since
it is easier to make the transition between cohomology and differential forms.\footnote{Note that if $X$ is a compact manifold and $M$ a $\Z[\pi_1 X]$-module free over $\Z$, Poincar\'e duality implies that $H^* (X ;M) \cong H_{\dim X - *} (X;M)$.}

 Section \S \ref{myCM} is expository: a review of analytic and Reidemeister torsion.
 We discuss in detail the trivial example of cyclic covers of a circle (equipped with a nontrivial local system); although simple, many of the ideas we use can be already seen clearly here. 
 
 \S \ref{heatkernel1} and \S \ref{heatkernel2} proves \eqref{mainresult-sharp}, i.e.
  that the limit of Theorem \ref{main}
 exists  when one uses an alternating $j$-sum. 
  
 \S \ref{explicit} is a generalization of the work of Olbrich \cite{Olbrich} to the case
of an arbitrary local system.  In particular,  we compute explicitly the constant $c_{G,M}$ of \eqref{mainresult-sharp}
and verify that it is {\em positive.} (We do not see a simple way to check this; we compute each case.) 
 
In \S \ref{Ash}, we discuss the conjectures of Ash and their relationship to 
Theorem \ref{main}. 

  \S \ref{cp} has a different flavour from the rest of the paper; it studies a tower of coverings
  of an arbitrary cell complex. We show, in particular, that a suitable torsion quotient
  grows at a consistent rate as one passes through cyclic covers of a fixed base manifold, generalizing slightly the result of \cite{SilverWilliams}. 
  
 \S \ref{examples} details some instances to which Theorem \ref{main} is applicable. 
In particular, {\em strongly acyclic bundles always exist} in the setting of Theorem \ref{main}, 
when $\delta = 1$. 
 
 Finally, \S \ref{s:quest} outlines various questions and conjectures motivated by the rest of the paper.   Several of these questions seem amenable to numerical study. 
 
\subsection{Acknowledgements.} The first author (N.B.) would like to thank Pierre Lochak who taught him goodness and 
Julien March\'e who told him about Silver and Williams' theorem.
 
The second author (A.V.)  would like to express his deep gratitude to Frank Calegari
for many conversations related to the subject material. Besides this, Calegari's ideas have played an important role in stimulating the present work: The idea
that torsion should grow fast was suggested by the unpublished data of Calegari--Dunfield
as well as work of Calegari--Emerton. 
In addition, the use of analytic torsion in such arithmetic questions
was already used in the joint work with Calegari \cite{CV}.
A.V. also expresses his gratitude to Universit{\'e} Pierre et Marie Curie for their kind hospitality in the summer of 2008, to Laurent Clozel who suggested that it would be interesting to study torsion in the homology of
arithmetic groups, and to Avner Ash for encouragement and helpful remarks. 

Finally, we would like to thank Jordan Ellenberg, Matt Emerton and Werner Mueller for their comments
on, and corrections to,  a draft form of this paper.  

\section{Reidemeister torsion, torsion homology, and regularized determinants.} \label{myCM}

This section is expository: 
a recollection of works by Ray, Singer, Cheeger, and M{\"u}ller, as well
as some remarks related to $L^2$-torsion.  We have attempted to present it in a way
that motivates our main questions and theorems.  The reader with some experience
with torsion (analytic and Reidemeister) could likely skip this section.

  In other words, we shall discuss:
\begin{enumerate} \item  {\em Analytically}  computing homological torsion for a Riemannian manifold;
\item {\em Limiting behavior} as the manifold grows (say, through a tower of covers). 
\end{enumerate}
We do not present the material in generality, but only adapted to our case of interest. 

\subsection{}

Given a finite rank free $\Z$-module $A$, so that $A \otimes \R$ is endowed with a positive definite
inner product $\langle \cdot , \cdot \rangle$ (a {\em metric} for short),
we define $\vol(A)$ to be the volume of $A \otimes \R/A$, i.e. 
$\vol(A) =\sqrt{|\det M|}$, where 
$M$ is the Gram matrix $\langle a_i, a_j \rangle$, for $a_i \ : \ 1 \leq i \leq \mathrm{rank}(A)$  a $\Z$-basis for $A$. 

Given $f: A_1 \rightarrow A_2$ a linear map between two finite rank free $\Z$-modules, so that 
$A_j \otimes \R$ are endowed with metrics, we set $\det'(f)$ to be the product of all nonzero singular values of $f$.  (Recall that the nonzero singular values of $f$ are -- with multiplicity -- the positive square roots of the nonzero eigenvalues of $f f^*$.) 
Then: 
\begin{equation} \label{erasmus} \frac{\vol(A_1)}{\vol(\mathrm{ker} f)}  \detp(f)  = \vol(\mathrm{image} f).\end{equation}

Here we understand the metrics on $\mathrm{ker}(f) \otimes \R$ and $\mathrm{image}(f) \otimes \R$ as those induced from $A_1$ and $A_2$, respectively. 

\subsection{}

Now given a complex 
\begin{equation} \label{abullet} \Abull : 0 \rightarrow A^0 \stackrel{d_0}{ \rightarrow} A^1 \stackrel{d_1}{\rightarrow} \dots \rightarrow A^n \rightarrow 0 \end{equation} 
of free finite rank $\Z$-modules, each endowed with metrics and with $\vol(A^j)=1$, the previous result says
\begin{equation} \label{prv}  \vol(\ker(d_{j})) \vol(\image(d_{j})) = \detp(d_j). \end{equation}
The quotient $ \frac{ \vol(\ker d_j) } {\vol ( \image  \ d_{j-1} )}$
is clearly the inverse of the size of the group $H^j (\Abull)$ if this group is finite; in general, it is the product of $|H^j (\Abull)_{\tors}|^{-1}$
with the ``regulator''
$R^j(\Abull) := \vol(H^j(\Abull)_{\free}),$
where the subscript $\free$ denotes quotient by torsion, and 
where the volume is taken
with respect to the induced metric -- that is to say,  the metric induced by identifying
$H^j(\Abull)_{\free}$ as a subgroup of 
$\image(d_{j-1})^{\perp} \cap \ker(d_j)_\R$. 

In summary, 
\begin{multline} \label{volquot} \frac{ \vol(\ker d_j) } {\vol ( \image  \ d_{j-1} )} = 
|H^j_{\tors}|^{-1} \cdot R^j, 
 \\ H^j = H^j(A^{\ast}) =  \frac{\ker d_j }{ \image \  d_{j-1}}, \ R^j =  R^j(\Abull) := \vol(H^j(\Abull)_{\free}).
\end{multline} 

To ease notation in what follows, we shall use the notation $\prod^* a_i$ as shorthand
for $\prod_i a_i^{(-1)^i}$. In the same way we will use the notation $\sum^* a_i$ for the alternating sum $\sum (-1)^i a_i$.

Taking the alternating product of \eqref{volquot}
\begin{equation} \label{RT} 
\left(\prodalt R^i  \right) \times \left( \prodalt |H^{i}_{\tors}| 
\right)^{-1} = \prodalt \detp (d_i). 
\end{equation}

\medskip

\noindent
{\it Example.} 
Consider the complex of free $\Z$-modules, where the first nonvanishing term is in degree $0$:
$$0 \rightarrow \Z^2 \stackrel{M}{\rightarrow} \Z^2 \rightarrow 0$$
where 
$$M = \left( \begin{array}{cc}
k^2 & -k \\
-k & 1
\end{array} \right) \ \ \ (k \in \Z).$$ 
Then $H^0 \cong H^1 \cong \Z$; 
then $R_0 = \sqrt{k^2+1}$, $R_1 = \frac{1}{\sqrt{ k^2+1}}$; the singular values of $M$
are $0$ and $k^2+1$.    
\medskip

The alternating product $\prod^* \det'(d_i)$ is, in many ways,  unwieldy: the $d_i$ may have large kernels 
even if the complex is acyclic. On the other hand, the identity:
\begin{equation} \label{dLap}  \left( \prod \- {}^* (\detp d_i)  \right)^2= \prod_{i \geq 0} (\detp \Delta_i)^{i (-1)^{i+1}}, \ \Delta := d d^* + d^* d, \end{equation}  -- 
here $\detp$ is the product of all nonzero eigenvalues; this usage is compatible with our prior one -- 
expresses it in terms of the determinants of the ``Laplacians'' $\Delta_i$, whose kernels
map isomorphically to cohomology (and in particular, are much smaller than $\mathrm{ker}(d_i)$).

\subsection{} Let $X$ be a compact Riemannian manifold and $M$ a unimodular local system of free $\Z$-modules on $X$.  Let $M_{\C} = M \otimes \C$.   

Fix a $C^{\infty}$ triangulation of $K$.  Let $C^q(K; M)$ be the set of $M$-valued cochains, so that
an element of $C^q(K;M)$ assigns to each $q$-cell a section of $M$ on that cell. 
Then we have a cochain complex:
$$C^{\ast}(K;M): C^0(K;M) \rightarrow C^1(K;M) \rightarrow \dots $$
which computes the cohomology groups $H^{\ast}(K;M) = H^{\ast} (X; M)$. 
The {\em de Rham complex}
\begin{equation} \label{DRC} \Omega^{\ast}(X;M) : \Omega^0(X;M_{\C}) \stackrel{d^{DR}_0} \rightarrow \Omega^1(X;M_{\C}) \rightarrow \dots\end{equation}
computes also $H^{\ast}(X;M_{\C})$. It can be regarded as a ``limit'' of the prior complex when 
the triangulation becomes very fine.  We fix arbitrarily a metric on $M_{\C}$;
then each term of the complex $\Omega^{\ast}$ is equipped with a natural inner product. 

The amazing discovery of Ray and Singer (as proved by Cheeger and M{\"u}ller \cite{Cheeger, Muller1} for trivial coefficients, and later by M{\"u}ller \cite{Muller2} for general unimodular $M$)
is that the equality \eqref{RT}, when applied to $C^{\ast}(K;M)$, ``passes to the limit'' in the following sense: 
\begin{eqnarray} \label{Muller}
\left(\prodalt R^i \right) \times \left( \prodalt | H^{i} (X; M)_{\tors}|^{-1} \right) = \prodalt \detp d^{DR}_i,
\end{eqnarray}
where $R^i$ is now defined as the volume of $H^{i} (X;M)/ H^i (X; M)_{\tors}$ with respect to the metric
induced by identifying cohomology with harmonic forms. 
To make sense of $\detp d^{DR}_i$, write $\zeta(s) = \sum \lambda^{-s}$, the sum 
being extended over all nonzero singular values of $d^{DR}_i$; 
we then define $\log \detp d^{DR}_i :=  -\zeta_{\lambda} ' (0)$
and $\detp d^{DR}_i = \exp(-\zeta_{\lambda}'(0))$. 

More generally, for any countable subset $S$ of positive reals so that $\sum_{\lambda \in S} \lambda^{-s}$ extends to a  meromorphic function in $\Re(s) \geq 0$,
we refer to the outcome of this procedure as the ``$\zeta$-regularized product,''
denoted $\prod_{\lambda \in S} ^{\wedge} \lambda.$

This is usually formulated slightly differently. 
Let $d_i^*$ denote the {\it adjoint of the differential} $d_i = d_i^{DR}$.
Define the Laplace operators on $i$-forms by $\Delta_i = d_{i}^* d_i + d_{i-1} d_{i-1}^*$. 
Again, one can define the logarithmic determinant of $\Delta_i$ as the zeta-regularized product
of nonzero eigenvalues.

Then we have the following, easily verified, identity:
\begin{eqnarray} \label{antor}
\sum \- {}^* \log \detp d_i = \frac12 \sum_{k \geq 0} (-1)^{k+1} k \log \detp \Delta_k .
\end{eqnarray}
Ray and Singer have defined the logarithm of the {\it analytic torsion} of $(X,M)$ to be the 
negative of the right hand side of \eqref{antor}.

In summary, the Cheeger-M{\"u}ller theorem states:
\begin{eqnarray*} \label{Muller}
\log \frac{\left(\prodalt R^i \right)}{\left( \prodalt | H^{i} (X; M)_{\tors}| \right)} &=& \frac12 \sum_{k \geq 0} (-1)^{k+1} k \log \detp \Delta_k  \\  &=& - \  \mbox{analytic torsion.}\end{eqnarray*}
which has the remarkable consequence that {\em torsion in homology groups can be studied by analytical methods.}

\subsection{} \label{circle} Consider a simple example: $X = S^1 = \mathbb{R}/\mathbb{Z}$, endowed with the quotient of the standard metric on $\R$; let $M$ be a free rank $m$ $\Z$-module, $A \in \SL(M)$, 
and $\mathcal{M}$ the local system on $X$ with fiber $M$ and monodromy $A$. 

Suppose first that $A$ is semisimple and does not admit $1$ as an eigenvalue.  
(The situation where $A$ has $1$ as an eigenvalue is interesting; we return to it in \S \ref{Reg1}).
The cohomology is then
concentrated in $H^1$; moreover \begin{equation} \label{tors1} |H^1(X, \mathcal{M})| = |\det(1-A)|.\end{equation}

Let us now compute the de Rham complex. Let $V_j$ be the space of $M _{\C} := M \otimes \C$-valued $j$-forms on $S^1$, for $j=0,1$.  
Let 
$$ S = \{  \mbox{smooth functions }f: \mathbb{R} \rightarrow M_{\C}: f(x+1) = A f(x)\} $$
and 
$$T = \{   \mbox{smooth functions } f: \mathbb{R} \rightarrow M_{\C}: f(x+1) = f(x)\}.$$

We may identify each $V_j$ with  $S$: this is obvious for $j=0$; use the map $f(x) \mapsto f(x) dx$ for $j=1$. 
Next, fix a matrix $B \in \GL(M_{\C})$ with $\exp(2 \pi i B) = A$; this is always possible, and then
multiplication by $\exp(-2 \pi i x B)$ identifies $S$ with the space $T$.

Thus we have identifications of $V_j$ with $T$ for $j=0,1$.  With respect to these identifications, the de Rham complex becomes
$$T \stackrel{d}{\rightarrow} T \ \ \ \mbox{ with } \ \ \   d = \frac{d}{dx}+ 2 \pi i B.$$
Fix an basis of eigenvectors $v_1, \dots, v_m$ for $B$ on $M_\C$,
and endow $M_\C$ with the inner product in which the $v_i$ form an orthonormal basis;
endow $T$ with the inner product $\langle f, g \rangle = \int_{x \in \R/\Z} \langle f(x), g(x) \rangle.$ 
Then the singular values of $d$ are  $| 2 \pi i( n + \lambda_j)|, n \in \Z,  \ \ 1 \leq j \leq m,$
where $\lambda_j$ are the eigenvalues of $B$. 
Formally speaking, then, the product of singular values of $d$ may be expressed as $\prod_{n} 
\prod_j |2 \pi n + 2 \pi \lambda_j|$.  If we compute formally, using the identity $\prod_{n \geq 1} (1-\frac{x^2}{n^2}) = \frac{\sin (\pi x)}{\pi x}$, we arrive at:
$$ |\detp d_{DR}|  =  |\det(1-A)| \left( \prod_{n \geq 1} (2 \pi n)  \right)^2.$$

Now  $\prod_{n \geq 1} {2 \pi n} = 1$ (zeta-regularized product), as desired. 
 In fact,  $\zeta_{\lambda} (s) = (2\pi)^{-s} \zeta (s)$ with $\zeta$ the Riemann zeta function. Because of $\zeta' (0) = -\frac12 \log (2\pi)$ and $\zeta (0) = -\frac12$ we get
$$\prod_{n \geq 1} \- {}^{\wedge} (2\pi n) = 1.$$

\medskip
\noindent
{\it Remark.}
Based on the naive (but natural) idea of considering
the asymptotics of $\prod_{\lambda_j \leq T} \lambda_j$,
one can also give an alternate definition of the regularized product: 
Let $h$ be a smooth function on $\R$ of compact support so that $h(0)= 1$. Set 
$P(T) := \sum_{\lambda_i} h(\lambda_i/T) \log \lambda_i.$
Under reasonable conditions on the $\lambda_i$ (true in all the instances we will encounter, for instance) there exists a unique polynomial $q \in \C[X,Y]$ so that
\begin{equation} \label{asymp} P(T) =q(T, \log T) +o(1) ,\ \ T \rightarrow \infty,\end{equation}
and we {\em define} $\prod^{\wedge} \lambda_i =\exp(q(0,0))$. 
This will coincide with $\zeta$-regularization in the cases we encounter. 

This definition is not the traditional one, but, in most instances, is equivalent to it. 
It has the advantage of being very intuitive, but, on the other hand, it is difficult
to work with. Since we will not use it for rigorous proofs, only for intuition, we shall not prove that it agrees with the zeta-regularized definition. 

In the case of the $(2 \pi n)_{n \in \mathbb{N}}$, the fact that $\prod {} ^{\wedge} 2\pi n = 1$
``corresponds'' in this picture to the fact that there is no constant term in the asymptotic 
of Stirling's formula:
$$\log \prod_{n \leq N-1} (2 \pi n)  \cdot \sqrt{2 \pi N} \sim  N \log N - N;$$
we included the endpoint $n =N$ with weight $1/2$, a simple form of smoothing.

\medskip

\subsection{} \label{circle2}
We continue with the example of \S \ref{circle}, and let $X \stackrel{\pi_N}{\rightarrow} X$
be the $N$-fold covering.   If we suppose that $A$ has no eigenvalue that is a root of unity, then
\begin{equation} \label{tors} \frac{\log |H^1(X, \pi_N^* \mathcal{M}) | }{N}= \frac{\log |\det(1-A^N)|}{N} \longrightarrow \log M(\mathrm{charpoly}(A)),  \end{equation}
where the Mahler measure $M(q)$ of a monic polynomial $q$
is the product of all the absolute values of all roots of $q$ outside the unit circle; equivalently,
$\log M(q) = \int_{|z|=1} \log |q(z)| \frac{dz}{2 \pi i z} $ (the integral taken with respect to the Haar probability measure).     We shall later return to the case where $A$ has roots of unity;
we shall find that a similar assertion remains valid.

The proof of \eqref{tors} uses:
\begin{lem*}  If $A$ is any square matrix, 
\begin{equation} \label{gelf} \lim_{N} \frac{\log \det'(1-A^N)}{N} = \log M(\mathrm{charpoly}(A)),\end{equation}
where $\det'$ denotes the product of all nonzero eigenvalues. 
\end{lem*}

\proof 
The first proof in print of \eqref{gelf} seems to be by Lind \cite{Lind} in the context of
a dynamical interpretation. The proof uses the fact (due to Gelfond \cite{Gelfond}):
if $\alpha$ is an algebraic number of absolute value $1$, then 
\begin{equation}  \label{mahler}
\frac{  \log |\alpha^N - 1| }{N} \rightarrow 0, \ \ N \rightarrow \infty.
\end{equation}
We refer to \cite[Lemma 1.10]{EverestWard} for a proof of \eqref{mahler} which uses
a deep theorem of Baker on Diophantine approximation. As noticed there, the statement 
\eqref{mahler} is much weaker than Baker's theorem;  it is exactly equivalent to Gelfond's earlier
estimate.
\qed

\subsection{The case when $A$ has a trivial eigenvalue}
\label{Reg1} 

Now let us redo \S \ref{circle2} in the case where $A$ has a trivial eigenvalue. We suppose, for simplicity, that precisely one eigenvalue of $A$ is $1$, and amongst the other eigenvalues $\{z_2, \dots, z_n\}$
none is a root of unity. 
We are interested in the size of
$ H^1(X, \pi_N^* \mathcal{M}) \cong  \mathrm{coker}(1-A^N)_{\tors},$
as $N \rightarrow \infty$.   

Write $W= \mathrm{ker}(1-A)$ and $U = \mathrm{image}(1-A)$;
finally put $C = \frac{1-A^N}{1-A} |_U$. Noting
$U/CU \subset \Z^N/(1-A^N) \Z^N,$
we see at least that: 
$$ |\mathrm{coker}(1-A^N)_{\tors} | \geq |\det C| = \prod_{2}^{n} \frac{z_i^N-1}{z_i-1};$$

Consider now the short exact sequence $W \hookrightarrow \Z^n \rightarrow \Z^n/W$.
The map $1-A^N$ acts trivially on $W$, and descends to an endomorphism $A'$ of $\Z^N/W$.
Applying the snake lemma yields an exact sequence:
$$0 \rightarrow \Z \rightarrow \mathrm{coker}(1-A^N) \rightarrow
\mathrm{coker}(1-A'^N) \rightarrow 0,$$
which implies the upper bound $\prod_{2}^{n} (z_i^N-1)$. 

We deduce that \eqref{tors} {\em remains valid even if $A$ has $1$ as an eigenvalue}, i.e.,
even in the case when $H^1(S^1, \mathcal{M})$ is not a torsion group. 

 In general, the question of extending our results to the case when the homology is not purely torsion will be very difficult;  the ``interaction'' between characteristic zero cohomology and torsion can be very complicated. {\em Implicitly} in the above argument, we used the possibility of ``splitting''
into $\mathrm{ker}(1-A)$ and $\mathrm{image}(1-A)$; more generally, we can carry out such a ``splitting'' in the presence of a group of automorphisms, but not in the general setting. 
We discuss some of the issues involved (including extending the discussion above to a more general setting -- that of cyclic covers of an arbitrary base complex) in \S \ref{cp}.

\subsection{}
Let us now try to prove \eqref{tors} from the ``de Rham'' perspective, returning to the case where
$A$ has no eigenvalue that is a root of unity. 

If we carry out the same computations as before for the determinant of the de Rham complex
of $\pi_N^* \mathcal{M}$ -- {\em with respect to metrics pulled back from $X$} -- we arrive at:
$$|H^1(X,  \pi_N^* \mathcal{M})_{tors}| =  \prod_{n \in \mathbb{Z}} \prod_j {}^{\wedge} |\frac{2 \pi i n}{N} + 2\pi i  \lambda_j| $$

Now as $N \rightarrow \infty$ the sequence of numbers $\frac{2 \pi n}{N} + 2\pi \lambda_j$
fill out densely the line $2\pi\lambda_j +  \mathbb{R}$; the ``density'' is $N/2 \pi$. 
It is therefore reasonable to imagine that it converges to a regularized integral: 
\begin{equation} \label{unjustified} 
\frac{ \log \prod_{n \in \mathbb{Z}}{}^{\wedge} |\frac{2 \pi i n}{N} + 2\pi i  \lambda_j|  }{N/2\pi}  \stackrel{?}{\longrightarrow}  \int^{\wedge} \log|2 \pi \lambda_j +  x| dx, 
\end{equation} 
where the wedge on the latter integral means that it is to be understood, again, in a regularized fashion:
analogously to $\zeta$-regularization, we may define the regularized integral via:
\begin{equation} \label{int-reg}
 \int^{\wedge} \log |f(x)|dx := 
-\frac{d}{ds}\big|_{s=0} \int | f(x)|^{-s} dx.\end{equation}   We evaluate the right hand side of \eqref{unjustified} (and a slight generalization) in \S \ref{an-integral} below;
it is $ 2 \pi^2 | {\rm Im} \lambda_j|$.  
Therefore, if we assume the validity of \eqref{unjustified}, 
$$\frac{\log |H^1(X, \pi_N^* \mathcal{M})_{tors}| }{N} \longrightarrow \pi \sum_j |{\rm Im}(\lambda_j)|
=2 \pi \sum_{{\rm Im} (\lambda_j) >0} {\rm Im}(\lambda_j) ,$$ 
which is easily seen to again equal $ \log M(\mathrm{charpoly}(A))$. 

The validity of  \eqref{unjustified} is, however, not easy to justify in general, as perhaps might be expected
from the nontriviality of \eqref{mahler}. The issue here is that some $\lambda_j$
may have imaginary part $0$ (equivalently: $A$ has an eigenvalue of absolute value $1$). 
Suppose this is not so.  Then  
\eqref{unjustified} can be deduced, for instance, from the ``naive definition''  \eqref{asymp} of regularized products. The key point is that
the convergence of the sequence  $\frac{2 \pi n}{N} $ 
to the density $\frac{N}{2 \pi} dx$ is uniform, in 
that \begin{equation} \label{planch-simple}\# \{n \; : \; 2 \pi n/N \in [A,B] \} - \int_{A}^{B} \frac{dx}{2 \pi} =o(N),\end{equation}
in a fashion that is {\em uniform in $A,B$.}

In the rest of this paper, we shall encounter a similar, although more complicated, situation: we will compute the asymptotic growth of $H^1$ 
by studying the asymptotic behavior of the de Rham determinants.
It would be possible to justify our main theorems
in the fashion described above -- in other words, {\em
deducing the limiting behavior of torsiuon from the limiting distribution of eigenvalues.}
However, for brevity of proof  we follow a more traditional path with heat-kernels; this allows us to appeal to a family of known estimates at various points.

\subsection{}  \label{an-integral} 
Let $a > 0$ be a positive real number and $p$ be an even polynomial.  We show:
\begin{eqnarray} \label{intp}
\int^{\wedge} p(ix) \log(x^2+a^2) dx = \pi \int_{-a}^{a} p(x) dx,
\end{eqnarray}
where the left-hand side is defined as in \eqref{int-reg}, but replacing
Lebesgue measure by $p(ix) dx$. 

\proof 
By linearity it suffices to check $p(x) = x^k$ with $k$ even. In this case~\footnote{Take as a new
variable $t= x^2 /(x^2+a^2)$.}, 
$$\int_{-\infty}^{\infty} (x^2+a^2)^{-s} x^k dx =    a^{2( \frac{k+1}{2} -s)} B\left( \frac{k+1}{2} , s - \frac{k+1}{2} \right) .$$
Here $B$ is the Beta function 
$$B(x,y) = \int_0^1 t^{x-1} (1-t)^{y-1} dt = \frac{\Gamma (x) \Gamma (y)}{\Gamma (x+y)}.$$
Differentating the function
$$a^{2( \frac{k+1}{2} -s)}  \frac{  \Gamma(\frac{k+1}{2}) \Gamma(s-\frac{k+1}{2}) }{ \Gamma(s) }
= a^{2( \frac{k+1}{2} -s)} s  \frac{  \Gamma(\frac{k+1}{2}) \Gamma(s-\frac{k+1}{2}) }{ \Gamma(s+1) }$$
at $s=0$ and using that 
$$\Gamma \left(\frac{k+1}{2} \right) \Gamma \left(-\frac{k+1}{2} \right) = (-1)^{k/2 +1} \frac{2\pi}{k+1}$$ yields $- i^{k} \frac{2 \pi a^{k+1}}{k+1}$.  
 \qed

Again, this reflects the constant term of a truncated integral, e.g. we observe, in the case $p=1$, that $\int_{-T}^{T} dx \log(1+x^2) = 
 4 T \log(T) - 4 T + 2 \pi +o(1)$.

\section{Heat kernels on forms on a symmetric space} \label{heatkernel1}

\subsection{} 

Let $\mathbf{G}$ be a semisimple algebraic group over $\mathbb{R}$;
let $G$ be the connected component of $\mathbf{G}(\R)$, with Lie algebra $\mathfrak{g}$.

Let $K \subset G$ be a maximal compact subgroup, with Lie algebra $\k$.  
Since $G$ is connected, so also is $K$. 
We denote by $S$ the global Riemannian symmetric space $G/K$.

{\em Remark.} (Notation for real and complex Lie algebras.) 
When we perform explicit evaluations in \S \ref{explicit}, we shall use the letters $\mathfrak{g}, \mathfrak{k}$ and so on 
to denote the {\em complexified} Lie algebra associated to $G, K$ and so on.  However, in the present section,
we use them for the real Lie algebra; we hope this abuse does not cause confusion. 

\subsection{} \label{vb}
To a $K$-representation $(Q, \sigma)$, we may associate a $G$-equivariant vector bundle on $G/K$, namely, $(Q \times G)/K$, where the $K$-action (resp. $G$-action) is given by
$(q,g) \stackrel{k}{\rightarrow} (\sigma(k^{-1}) q, gk)$ (resp. $(q,g) \stackrel{x}{\rightarrow} (q, xg)$);
thus, smooth sections of the $G$-equivariant vector bundle associated to $Q$
are identified with maps $C^{\infty}(G;Q)$ with the property that $f(gk) = \sigma(k^{-1}) f(g)$.

Conversely, given a $G$-equivariant vector bundle on $G/K$, the fiber at the identity coset
defines a $K$ representation.

These two associations define an equivalence of categories between $K$-representations
and $G$-equivariant vector bundles on $S$.

\subsection{}  \label{casimirs} 

We now discuss normalizations.

The Killing form defines a quadratic form on $\mathfrak{g}$, negative definite on $\mathfrak{k}$
and positive definite on its orthogonal complement $\mathfrak{p}$.  We normalize the metric on $S$ so that its restriction to $\mathfrak{p} \cong T_{eK} S$ matches the Killing form.

Let $dx$ be the Riemannian volume form of $S$ and $dk$ be the Haar measure on $K$ of volume $1$. 
These choices yield a preferred normalization for the Haar measure $dg$ on $G$:
$\int_G f(g) dg = \int_X \int_K f(gk) dk d(gK)$.

Fix orthonormal bases $\{X_i \}$ for $\mathfrak{p}$
and $\{Y_j \}$ for $\mathfrak{k}$ (the latter taken with respect to the negative of the Killing form). 
We define the Casimir operators of $G$ resp. $K$ by $\Omega_G = \sum X_i^2  - \sum Y_j^2$
and $\Omega_K = - \sum Y_j^2$; we understand these to act on $G$ and $K$ as
{\em left}-invariant differential operators. 
Let $$L = -\Omega_G + 2 \Omega_K \in {\cal U} (\mathfrak{g}).$$

 The operator $L$ on $G$ is the infinitesimal generator of a convolution semigroup of absolutely continuous measures on the group $G$.  Indeed, the
right action of $L$ coincides with the Laplacian with respect to a suitable metric on $G$. 

There exists, for positive $t$,  a function $p_t : G \rightarrow {\Bbb R}_{\geq 0}$,
so that $e^{-tL} f = f * p_t$, i.e., the function $g \mapsto \int_{h \in G} f(gh) p_t(h) dh.$
In fact, on any $G$-representation, the action of $p_t$ by convolution coincides
with the action of $e^{-tL}$.

\subsection{} 
 Put $\gcompact = \k \oplus i \p$; its normalizer inside $G_{\C} := \mathbf{G}(\C)$ is a maximal compact subgroup $U$ of $G_{\C}$. 
Let $\rho$ be an irreducible  representation of $U$;  it extends to a unique holomorphic representation of $G_{\C}$ on a complex
vector space $E_{\rho}$, or indeed an algebraic representation of $\mathbf{G}$. 
Since $\mathbf{G}$ is semisimple, the representation $\rho$ is necessarily {\em unimodular}, 
i.e. takes values in $\SL(E_{\rho})$. 

There is, up to scaling, one $U$-invariant Hermitian metric on $E_{\rho}$.  We fix an inner product $(-, - )_E$ in this class. 
 
The representation $\rho|K$ gives rise to a $G$-equivariant hermitian vector bundle on $G/K$, as in
\S \ref{vb}. We denote this vector bundle also as $E_{\rho}$.  Note that
this bundle is $G$-equivariantly isomorphic to the trivial vector bundle
$E_{\rho} \times G/K$, where the $G$-action is via $x: (e, gK) \longrightarrow (\rho(x) e, xgK)$.

\subsection{} \label{erhodiff}
The bundle of $E_{\rho}$-valued  differential $k$-forms on $S$
can be identified  with the vector bundle associated (via \S \ref{vb}) with the $K$-representation
$\DkE$.

Note that $\DkE$ is naturally endowed
 with a $K$-invariant scalar product: the tensor product of $(,)_E$ with the scalar product on $\bigwedge^k \mathfrak{p}^*$ defined by the Riemannian metric on $S$. The space
 of differentiable $E$-valued $k$-forms on $S$, denoted $\Omega^k(S, E)$, 
is therefore identified with
$$\left\{ \varphi : G \stackrel{C^{\infty}}{\rightarrow} \bigwedge \- {}^k \mathfrak{p}^* \otimes E \; : \; 
 \varphi(gk) = \bigwedge^k {\rm ad}_{\mathfrak{p}}^* (k^{-1}) \otimes \rho (k^{-1}))(\varphi (g)), \ \ \ g \in G, \ k \in K 
\right\}.$$

The space of square integrable $k$-forms $\Omega^{k}_{(2)} (S, E)$
is the completion of the latter space with respect to the norm
$$\varphi \mapsto \int_{S} ||\varphi (xK)||_{\DkE}^2 dx.$$

\subsection{The de Rham complex.} 

Let $V$ be any $(\mathfrak{g},K)$-module equipped with a $K$-invariant
Hermitian form; we do not require this form to be $G$-invariant.   We will be particularly interested
in $\Cinf = C^{\infty}_c(G)_K \otimes E$,   where  $C_c^{\infty}(G)_K$ denotes
right $K$-finite smooth compactly supported functions. 

Set $$D^k (V) = \bigwedge \- {}^k \mathfrak{p}^* \otimes V.$$ 
We introduce the natural inner product on $D^k(V)$.   Thus, for instance,  $D^k(V_0)^K$ is identified
with the space $\Omega^k(S,E)$ of $E$-valued $k$-forms. 

Generalizing the differential on $E$-valued $k$-forms, one may define for any $V$
a natural differential $d : D^k (V)  \rightarrow D^{k+1} (V)$; see \cite[II. Proposition 2.3]{BorelWallach}; in the case $V = \Cinf$, its restriction to $D^k(V_0)^K$
recovers the de Rham differential.  Let $d^*$ be the formal adjoint of $d$. We refer to the restriction of $dd^* + d^* d$ to $D^k(V)^K$ as the {\em Laplacian}. 

 In the case $V = \Cinf$, this ``Laplacian'' extends to a $G$-invariant self-adjoint non-negative densely defined elliptic operator $\Delta_k^{(2)}$ on $\Omega^{k}_{(2)}(S,E)$, the form Laplacian. \footnote{ It should be noted that there
is another notion of Laplacian, which is the specialization to $ \bigwedge^k \mathfrak{p}^* \otimes E$
of the Laplacian that exists on any Hermitian bundle with connection.  }

In general, the Laplacian is difficult to compute. However, the fact that the bundle $E_{\rho}$
arose from the restriction of a $G$-representation makes the situation simpler:  
if the Casimir acts on $V$ (resp. $E$)  through the scalar  $\Lambda_V$ (resp. $\Lambda_\rho$) 
Kuga's lemma implies that for any $f \in D^k (V)^K$,
\begin{equation} \label{Kuga} (d^* d+dd^*) f = (\Lambda_\rho - \Lambda_V) f. \end{equation} 
In other terms, the Laplacian acts by the same scalar on the entire complex $D^*(V)^K$. 

Of course, this does not literally apply to our case of primary interest, when $V=V_0$, for the Casimir does not act on $\Cinf$ as a scalar. However, one may decompose
$\Cinf$ into irreducible subrepresentations and then apply \eqref{Kuga}.

\subsection{}

Denote by $e^{-t \Delta_k^{(2)}}
\in \mathrm{End}(\Omega^k_{(2)}(S, E))$
the bounded operator  (cf. \cite{BarbaschMoscovici}) defined by
 the fundamental solution of the heat equation:
$$\left\{
\begin{array}{l}
\Delta_k^{(2)} P_t = -\frac{\partial}{\partial t} P_t , \ \ \ t>0 \\
P_0 = \delta
\end{array} \right.$$
where $\delta$ is the Dirac distribution.
 
It is expressed by an integral kernel  (the heat kernel) $e^{-t\Delta_k^{(2)}}  $, which
we may regard as a section of a certain bundle $\mathscr{E}_k$ over $S \times S$. Explicitly, 
this is the $G \times G$-equivariant bundle associated to the $K \times K$-representation
$\End(\DkE, \DkE)$; more explicitly, the fiber of $\mathscr{E}_k$ above $(x,y)$
is the set of homomorphisms from $E$-valued $k$-forms at $x$ to $E$-valued $k$-forms at $y$. 
Moreover:
$$(e^{-t\Delta_k^{(2)}} f) (x) = \int_{S} e^{-t\Delta_k^{(2)}} (x,y) f(y) dy, \ \ \ \forall f \in  \Omega^k_{(2)}(S,E).$$

\begin{lem} \label{Lhk}
Let $M \geq 1$. Then there exists a constant $c_1$ depending only on $G, \rho, M $ such that 
$$|| e^{-t \Delta_k^{(2)}} (x,y) || \leq c_1   t^{-d/2} \exp(-\frac{r^2}{5t})  ,  \ \ |t| \leq M,$$ 
where $x,y \in S$,  $r$ is the geodesic distance between $x$ and $y$, $d$ the dimension of $S$
and $|| \cdot ||$ is the metric induced by that on $E$. 
\end{lem}

The Lemma gives us only information about the ``small time'' dependence of the heat kernel. 
The ``long time'' dependence is more sensitive; for instance, it is quite different if there
exists $L^2$-harmonic forms (e.g., for $G = \SL_2(\R)$).

\begin{proof}

First we make an observation in a more abstract context:
If $(V, \pi) $ is any $G$-representation with Casimir eigenvalue $\Lambda_{\pi}$ and $(W, \sigma)$ any $K$-representation, the operator
$Q := \int_{k \in K} \pi(k) \otimes \sigma(k)  dk $ realizes the projection of $V \otimes W$
onto $K$-invariants $(V \otimes W)^K$; on the other hand, denoting
by $E_t \in \End(W)$ the operator $e^{t \Omega_K}$, 
we have for $t  > 0$

$$T = \pi ( p_t ) \otimes E_{2t} = \int_{G \times K}  (p_t(g) \pi (g)) \otimes  E_{2t}  \ dg$$
acts on $(V \otimes W)^K$ by $e^{t \Lambda_{\pi}}$.
(Here $p_t$ is as in \S \ref{casimirs}). 
Thus $Q \cdot T \cdot  Q$
certainly acts on $(V \otimes W)^K$ by the scalar $e^{t \Lambda_{\pi}}$.

On the other hand, expanding the expressions above, we see
that $Q\cdot T \cdot Q = \pi (\psi )$ where $\psi$ 
is the $\End(W)$-valued function on $G$ given by:
\begin{equation} \label{Bigop} 
\psi_{t}: g \mapsto  \int_{K \times K \times K }   p_t(k_1^{-1} g k_2^{-1})  \left(  \sigma(k_1  ) \cdot  E_{2t} \cdot  \sigma( k_2) \right) d \kappa d k_1 d k_2.
\end{equation}
Let us note that the norm of the right-hand side, with respect to the natural Hilbert norm on $\End(W)$, 
is bounded above by a constant multiple (depending on $M,\sigma$) of $\int_{K \times K} p_t(k_1^{-1} g k_2^{-1}) dk_1 dk_2$, 
since $p_t$ is {\em positive}. The integral $\int_{K \times K} p_t(k_1^{-1} g k_2^{-1}) dk_1 dk_2$ defines  a bi-$K$-invariant function on $G$
that is identified with the (usual) heat kernel on the symmetric space $S$; in particular, 
by the Cheng-Li-Yau bounds \cite{CLY}, it is bounded by  for $|t| \leq M$ by
$ c_1(M) t^{-d/2} \exp(-\frac{r^2}{5t})$, with $r$ the distance between $gK$ and the identity coset.

We now apply this construction with $V$ equal to the underlying $(\mathfrak{g} , K)$-module of the
{\em right} regular representation $L^2(G)$; and $W =  \bigwedge^k \mathfrak{p}^* \otimes E$.
In this case $V$ is no longer irreducible; however, breaking it into irreducibles and applying \eqref{Kuga}, we see
that the operator \eqref{Bigop} acts on $(V \otimes W)^K = D_k(V)^K$ 
in the same way as $e^{-t \Delta_k^{(2)}}  e^{t \Lambda_{\rho}}$, where $\Lambda_{\rho}= \rho (\Omega )$ is the Casimir eigenvalue 
of the $G$-representation $E$. 

In other words, if we identify 
 elements in $\Omega^k$ with functions on $G$ as specified in \S \ref{erhodiff}, we have:
$$(e^{-t\Delta_k^{(2)}} f) (g) = e^{-t \Lambda_{\rho}}  \int_G  \psi_t(g^{-1} g') f(g') dg', \ \ \ \forall f \in  \Omega^k_{(2)}(S,E).$$
Passing to the associated section of $\mathcal{E}_k$ gives the desired assertion. 
\end{proof}

\section{A ``limit multiplicity formula'' for torsion.} \label{heatkernel2}

We continue with the notations of the previous section; 
let $\Gamma$ be a cocompact torsion-free subgroup of $G$ and let $X = \Gamma \backslash S$. 
Set $V$ to be the vector bundle on $X$ induced by $\rho$ (i.e., the quotient of the total space
$E_{\rho} \times S$ by the $\Gamma$-action). 
 Let $\Delta_k$ be the Laplacian on $E_{\rho}$-valued $k$-forms on $X$.  

Note that $H^*(X; V)$ is isomorphic to $H^*(\Gamma;V)$. We say, accordingly, that $\rho$ is acyclic for $\Gamma$ when  $H^k (X; V) = 0$ for each $k$; 
equivalently, the smallest eigenvalue of each $\Delta_k$ is positive. 
We say that $\rho$ is {\it strongly acyclic} if there exists some uniform positive constant $\eta = \eta (G) > 0$ such that every eigenvalue of every 
$\Delta_k$ {\em for any choice of $\Gamma$} is $\geq \eta$. The remarkable fact is that there exists a large and interesting supply of strongly acyclic representations;
see \S \ref{examples}.  To check this, the following will be necessary:
\begin{lem} \label{SA}
Suppose that the isomorphism class of $d\rho: \mathfrak{g} \rightarrow \mathfrak{gl}(E_{\rho})$
is not preserved under the Cartan involution. 
Then $\rho$
  is strongly acylic. 
\end{lem}

\proof  Set $F=E_{\rho}^*$; it is a finite dimensional representation of $\mathfrak{g}$. 
It is known \cite[\S VI, Thm. 5.3]{BorelWallach} that, if the isomorphism class of $F$ 
is not fixed by the Cartan involution,
then $\rho$ is acyclic (for any $\Gamma$ as above).
In fact, one of the arguments explained there shows strong acyclicity, as we explain.

Suppose that $\pi$ is an irreducible unitary representation for which  
\begin{equation} \label{fd}
\Hom_K(\wedge^* \mathfrak{p} \otimes F, \pi ) \neq 0.
\end{equation}
We show
that there exists a positive constant $\varepsilon$ depending only on $F$ such that 
$$\Lambda_{F} - \Lambda_{\pi} \geq \varepsilon.$$
We proceed as in the proof of \cite[\S II, Prop. 6.12]{BorelWallach}. 

Fix compatible (in the sense of \cite[\S II, 6.6]{BorelWallach}) positive root systems $\Delta^+$ and $\Delta_k^+$ for 
  $\mathfrak{g} \otimes \C$  and $\mathfrak{k}\otimes \C$ and denote by $\rho$ and $\rho_k$  the corresponding half-sums of positive roots. Let $W$ be the Weyl group of $\mathfrak{g}_{\C}$ and
set $W^1 = \{ w \in W \; : \; w \Delta^+ \mbox{ is compatible with } \Delta_k^+ \}$. Write $\nu$
for the highest weight of $F$; note that $\Lambda_{F} = |\nu + \rho |^2 - |\rho |^2$.

The orthogonal Lie algebra $\mathfrak{so} ( \mathfrak{p}) $ has a natural representation 
$S= \mathrm{spin} (\mathfrak{p})$. Since the adjoint action of $\mathfrak{k}$ on 
$\mathfrak{p}$ preserves the Killing form defining $\mathfrak{so} (\mathfrak{p})$, there
is a natural map $\mathfrak{k} \rightarrow \mathfrak{so} (\mathfrak{p})$; so $S$ may be regarded
as a representation of $\mathfrak{k}$. The exterior algebra $\wedge^* \mathfrak{p}$
is isomorphic to one or two copies (according to the parity of the dimension of $\mathfrak{p}$) of 
$S \otimes S$. Since $S$ is isomorphic to $S^*$ as a $\mathfrak{k}$-module, it follows from
\eqref{fd} that there exists a representation of $\mathfrak{k}$ occuring in both $F \otimes S$ and $\pi \otimes S$. 
Now \cite[\S II, Lem. 6.9]{BorelWallach} implies that every highest weight of $F \otimes S$ is of the form 
$\frac12 (\mu + \theta \mu) + w \rho - \rho_k$ where $\mu$ is a weight of $F$ and 
$w \in W^1$. Using the expression relating the square of the Dirac operator $D$ to the Casimir operator (see \cite[\S II, Lem. 6.11]{BorelWallach}), and
the positivity of $D$, we conclude that there exists a weight $\mu$ of $F$ and an element $w$ of $W^1$ such that 
\begin{equation*} \label{dirac}
|\frac12 (\mu + \theta \mu) + w \rho |^2  \geq \Lambda_{\pi} + |\rho|^2.
\end{equation*}
This inequality is a refined version of Parthasarathy's Dirac operator inequality \cite[(2.26)]{P}. 
It implies that:
\begin{equation} \label{dirac2}
\Lambda_F - \Lambda_{\pi} \geq |\nu + \rho |^2 - |\frac12 (\mu + \theta \mu) + w \rho |^2.
\end{equation}
Now since 
$|\frac12 (\mu + \theta \mu ) +  w \rho |^2 \leq |\mu + w \rho |^2$, the left-hand side of \eqref{dirac2} is nonnegative: 
In fact $\mu + w \rho$ is a weight of the finite-dimensional representation of $\mathfrak{g}$
given by $F \otimes F_{\rho}$, where $F_{\rho}$ has highest weight $\rho$; thus
$$|\mu  + w \rho| \leq |\nu  + \rho|,$$
with equality if and only if there exists $u \in W$ with $\mu + w \rho = u(\nu+\rho)$. 
But since $w\rho$ is $\theta$-invariant we conclude that 
the left-hand side of \eqref{dirac2} vanishes if and only if 
$$|\nu+ \rho|^2 = \left|\frac12  \left( u (\nu+\rho)+ \theta u (\nu+\rho)\right)\right|^2.$$
This forces $u (\nu+\rho)$ to be $\theta$-fixed, i.e., the isomorphism class of $F$ to be $\theta$-fixed. 
\qed

From now on we assume that $\rho$ is strongly acyclic.

We denote by $T_{X}(\rho)$ the analytic torsion of $(X, E_{\rho})$, defined as (cf. \S \ref{myCM}):
$$\log T_{X}(\rho)   =  \frac12 \sum_{k \geq 0} (-1)^{k+1} k  \  \log \detp  \Delta_k;$$
we recall the definition of $\detp \Delta_k$ in \S \ref{bel} below. 
It follows from \eqref{Muller} and \eqref{antor} that, if $\rho$ is strongly acyclic and
we are given a local system $M$ of free $\Z$-modules with an isomorphism $M \otimes \C 
\stackrel{\sim}{\rightarrow} E_{\rho}$, we have: %
\begin{eqnarray} \label{TX=tor}
\log T_X (\rho) =  - \sum \- {}^* \log |H^i ( X; M)|.
\end{eqnarray}

\subsection{} \label{bel}
Although already discussed -- somewhat informally -- in \S \ref{myCM}, let us recall the definition
of $\detp \Delta_k$. 

The Laplacian $\Delta_k$ is a symmetric, positive definite, elliptic operator with pure point spectrum
$$0 < \lambda_1 \leq \lambda_2 \leq \cdots \rightarrow +\infty$$
and, writing $e^{-t \Delta_k} (x,y)  \ \ (x,y \in X) $ for the integral kernel representing the heat kernel on $k$-forms on $X$,
$${\rm Tr} e^{-t \Delta_k} = \int_X {\rm tr} (e^{-t\Delta_k} (x,x) )dx = \sum_{j=1}^{+\infty} e^{-t\lambda_j}$$
is convergent for each positive $t$.
We may thus define
$$\log \detp \Delta_k = - \frac{d}{ds} \zeta_k (s; \rho)_{| s=0}$$
where the function $\zeta_k$ is the unique meromorphic function \cite{Seeley}
satisfying
$$ \zeta_k (s; \rho) =  \frac{1}{\Gamma (s)} \int_0^{+\infty} t^{s-1} {\rm Tr} e^{-t\Delta_k} dt = \sum_{j=1}^{+\infty} \lambda_j^{-s}. $$
for ${\rm Re}(s)$ sufficiently large.

 \subsection{$\Delta_k$ and $\Delta_k^{(2)}$.} 
Our goal is to relate the analytic torsion on $X$ to computations on the universal covering $S$; in particular,
as the injectivity radius of $X$ tends to infinity, to approximate the former by the latter. 

Indeed, if we write $e^{-t \Delta_k} (x,y)$ $(x,y \in X)$ for the integral
kernel representing the heat kernel on $k$-forms on $X$, then 
for each positive $t$ and for each integer $i$ we have:
\begin{eqnarray} \label{sum}
e^{-t \Delta_k} (x,y) = \sum_{\gamma \in \Gamma} (\gamma_y )^* e^{-t \Delta_k^{(2)}} (\tilde{x}, \gamma \tilde{y}),
\end{eqnarray}
where $\tilde{x}, \tilde{y}$ are lifts of $x,y$ to $S$; by $(\gamma_y)^*$, 
we mean pullback by the map $(x,y) \mapsto (x, \gamma y)$. 
The sum converges absolutely and uniformly for $\tilde{x}, \tilde{y}$ in compacta: 
by cocompactness of $\Gamma$ and bounds on volume growth, 
there  exists constants $c_i = c_i(\Gamma, G)$, $i=2,3$, such that for any $x,y \in S$
\begin{eqnarray} \label{gauss}
N(x,y;R):=\left|\{ \gamma \in \Gamma \; : \; d(\tilde{x},\gamma \tilde{y}) \leq R \} \right| \leq c_2 e^{c_3 R},
\end{eqnarray}
which taken in combination with Lemma \ref{Lhk} gives the absolute convergence.
In particular, when the injectivity radius is big, a single term in \eqref{sum} will dominate.

\subsection{$L^2$-torsion.}
 
The trace of the ($S$-)heat kernel $e^{-t \Delta_k^{(2)}}(x,x)$ on the diagonal
is independent of $x$, because it is invariant under $G$. 
 We will prove in  \S \ref{explicit} that the integral 
$$\frac{1}{\Gamma (s)} \int_0^{+\infty} t^{s-1} \Tr  \ e^{-t\Delta_k^{(2)}} (x,x) dt$$
is absolutely convergent for ${\rm Re} (s)$ sufficiently large and extends to a meromorphic function of $s \in {\Bbb C}$
which is holomorphic at $s=0$.  
 
Define $t_S^{(2)}(\rho)$ via
\begin{equation} \label{L2analyticTorsion}
t_S^{(2)} (\rho) = \frac12 \sum_{k \geq 0} (-1)^k k \left( \frac{d}{ds} \Big |_{s=0} \frac{1}{\Gamma (s)} 
\int_0^{+\infty} t^{s-1}  \Tr \ e^{-t \Delta_k^{(2)}}(x,x)  dt \right).
\end{equation}
(The product of $t_S^{(2)}(\rho)$ with the volume of $X$
is the $L^2$-analytic torsion of $X$.)
It is possible to compute $t_S^{(2)}(\rho)$ in a completely explicit fashion -- it is an explicit quantity depending on $\rho$ -- and we carry this out in \S \ref{explicit}.  Thus the following Theorem shows that we can approximate the torsion
of any ``large'' quotient $X$ by an explicitly computable quantity.

\begin{thm} \label{approxthm}
Let $\rho : G  \rightarrow \GL (E)$ be a strongly acyclic representation. Then,
$$\frac{\log (T_{X} (\rho))}{{\rm vol} (X)} \rightarrow t_S^{(2)} (\rho)$$
when $\Gamma$ varies through any sequence of subgroups for which the injectivity
radius of $X = \Gamma \backslash S$ goes to $\infty$. 
\end{thm} 

We conjecture that this remains valid without the {\em strongly acyclic} assumption.  
We are unable to make significant progress without
ruling out {\em eigenvalues on $k$-forms very close to zero}. At a combinatorial level this seems to be equivalent to exceptionally difficult problems of diophantine approximation. 

The scheme of proof will be, when unwound, very close to the latter part of the discussion in \S \ref{circle2}.
In the language of \S \ref{circle2}, we already observed in that context that difficulties arose 
when ${\rm Im}(\lambda_j) = 0$, which manifested themselves
also arithmetically through \eqref{mahler}.  This is the issue of ``small eigenvalues''. 

\begin{proof} Let $\ell$ be the length of the shortest closed geodesic on $X$. Assume for simplicity that $\ell \geq 1$. 
Let 
$$A_k(t) = \frac{1}{{\rm vol} (X)} \int_{X} \left( \Tr \ e^{-t \Delta_k^{(2)}} (\tx,\tx) - \Tr \  e^{-t \Delta_k} (x,x) \right)dx,$$
where $\tx$ denotes any lift of $x$ to $S$ (in fact,  $ \Tr \ e^{-t \Delta_k^{(2)}}$ 
is constant, and so we could replace $\tx$ by a fixed basepoint in what follows). 

The following lemma is a slight modification of \cite[Lemma 4]{Lott}: 

\begin{lema} 
Let $M$ be a real number $\geq 1$. There exists a constant $c = c(M, \rho, G)$ such that for any $t \in (0, M]$,
\begin{equation} \label{upperest} 
\left|   A_k (t) \right| \leq c t^{-(d+1)/2} e^{-\frac{(\ell-ct)^2}{5t}}.
\end{equation} 
\end{lema}
\begin{proof} We may rewrite (\ref{sum}) as 
\begin{eqnarray*} \label{dif}
e^{-t\Delta_k } (x,x) - e^{-t \Delta_k^{(2)}} (\tx,\tx) = \sum_{\gamma \in \Gamma, \ \gamma \neq e}
 (\gamma_y )^* e^{-t \Delta_k^{(2)}} (\tx,\gamma \tx).
\end{eqnarray*}
It follows from lemma \ref{Lhk} and (\ref{gauss}) that, up to constants, the last sum is bounded by
\begin{eqnarray*}
\int_{\ell}^{+\infty} t^{-d/2} e^{-r^2/5t} d(N(x,x;r) - N(x,x;\ell)) \leq c_2 \int_{\ell}^{+\infty} \left| \frac{d}{dr} \left\{ t^{-d/2} e^{-r^2/5t} \right\} \right| e^{c_3 r} dr.
\end{eqnarray*}
And the last integral above is, up to a constant, bounded by
\begin{eqnarray*}
t^{-(d+1)/2} \int_{\ell}^{+\infty}  e^{-r^2/5t} r e^{c_3 r} dr & \leq & {\rm const} \cdot t^{-(d+1)/2} e^{-(\ell -3c_3t)^2/5t} \int_{\ell}^{+\infty} r e^{-c_3r/5} dr   \\
& \leq & {\rm const} \cdot t^{-(d+1)/2} e^{-(\ell -3c_3 t)^2/5t} .
\end{eqnarray*} 
Integrating (\ref{dif}) over $X$ thus gives the lemma. 
\end{proof}

It follows from \eqref{upperest} that 
$$\int_0^{+\infty } t^{s-1} A_k (t)  dt$$
is holomorphic in $s$ in a half-plane containing $0$, so that:
$$
\frac{d}{ds} \Big|_{s=0} \frac{1}{\Gamma (s)} \int_0^{+\infty} t^{s-1}A_k (t) dt 
= \int_0^{+\infty} A_k (t)  \frac{dt}{t}.
$$
It thus follows from the definitions that 
\begin{eqnarray} \label{difTor}
t_S^{(2)} (\rho) -  \frac{\log (T_{X} (\rho))}{{\rm vol} (X)}  = \frac12 \sum_{k=0}^{d} (-1)^k k
\int_0^{+\infty} A_k(t) \frac{dt}{t}.\end{eqnarray}
It is to handle the ``large $t$'' contribution to $\int_0^{\infty} A_k(t) \frac{dt}{t}$ that  ``strong acyclicity'' of $\rho$ enters. Let $\eta$ be so that every eigenvalue of $\Delta_k$ is $\geq \eta$.  Now, for any $t \geq1$, 
spectral expansion on the compact manifold $\Gamma \backslash S$ shows that:
$$ \trace e^{-t\Delta_k}(x,x) \leq e^{-\eta (t - 1)}   \ \trace e^{-\Delta_k }(x,x).$$
We apply \eqref{sum} and Lemma \ref{Lhk} with $M$ replaced by $1$, say,
to estimate the latter quantity, arriving at 
\begin{eqnarray} \label{upperest2}
\trace e^{-t \Delta_k(x,x)} \leq c' e^{-\eta(t-1)},
\end{eqnarray} 
where the constant $c'$ only depends on $\rho$ and $G$.
Thus, both 
$$\int_{1}^{\infty} t^{-1} e^{-t \Delta_k^{(2)}} (\tx,\tx) dt, \ \ \int_1^{\infty} t^{-1} e^{-t \Delta_k} (x,x) dt, $$
are absolutely convergent -- the former conclusion follows, for instance, from Lemma \ref{Lhk}. Thus, given $\nu > 0$ arbitrary, there exists a constant $M \geq 1$ depending only on $G$, $\rho$ and $\nu$ such that 
\begin{eqnarray} \label{L3}
\left| \int_{M}^{+\infty} t^{-1} A_k (t) dt \right| \leq \nu.
\end{eqnarray}
But it follows from \eqref{upperest} that 
$$\left| \int^{M}_0 t^{-1} A_k (t) dt \right| \rightarrow 0$$
as the injectivity radius $\ell$ tends to $\infty$. 
In particular, this shows that the right hand side of \eqref{difTor} approaches zero 
as the injectivity radius $\ell$ approaches $\infty$. 
\end{proof}

\section{Explicit evaluation of the $L^2$-torsion for unimodular local systems} \label{explicit}

In this section, we shall compute explicitly the asymptotic constants
of Theorem \ref{approxthm}.  The contents of this section are extensions of computations from 
\cite{Olbrich} and also relate to computations performed in \cite{CV}. 
 
\subsection{Notation} 

We continue with notation as in \S \ref{heatkernel1} {\em with the following change:} in what follows, if $H$ is a real Lie group, we will write $\mathfrak{h}$ for the {\em complexification} of its real Lie algebra, and $\mathfrak{h}_{\R}$ for the Lie algebra of $H$.

\subsubsection{Groups and subgroups}
We have already defined $K \subset G$ and $U$ a compact form of $G$. Note that we may identify the complexified Lie algebra $\mathfrak{u}$ of $U$ and the complexified Lie algebra of $\mathfrak{g}$. Let $S = G/K$ be the Riemannian symmetric space associated to $G$ and $S^c = U /K$ be its compact dual. 
Let $\Theta$ be the Cartan involution of $G$ fixing $K$. 

Fix a maximal torus $T_f \subset K$ with complexified Lie algebra $\bt$. 
Extend it to a $\Theta$-stable maximal torus $\mathfrak{t}_{U} = \bt \oplus  \at \subset \gcompact$,
where $\at \subset \p$ is the complexification of an abelian subspace $\atr \subset \p_{\R}$.  Therefore, $\bt \oplus  \at$ is a ``fundamental Cartan subalgebra,'' i.e. one with maximal compact part.  
We extend $\mathfrak{a}_{0\R}$ to an ``Iwasawa'' space $\a_{\R} \subset \p_{\R}$ -- that is to say, a maximal abelian semisimple subspace.

Let $T_U \subset U,  W_U \subset \mathrm{Aut}(\mathfrak{t}_U)$ be the maximal tori and Weyl groups that correspond to 
$\mathfrak{t}_U \subset \mathfrak{g}$.

Let $A_f = \exp(\atr)$. 
Then $A_f$ is the split part of a fundamental parabolic subgroup $P_f = M_f A_f N_f$;
here $M_f A_f$ is the centralizer in $G$ of $\at$, and the $\Theta$-stable subgroup $M_f$ is 
a certain canonical complement to $A_f$ (see \cite[Chapter V, \S 5]{Knapp}). 

 Let $K_f = K \cap M_f$; it is a maximal
compact subgroup of $M_f$; we denote by $M_f^0$ and $K_f^0$ 
the connected components of $M_f$ and $K_f$ respectively.  Denote by $\k_f$ and $\m_f$
the complexified Lie algebras of $K_f$ and $M_f$. 
Note that $T_f \subset M_f$ and $T_f \subset K_f^{0}$.

 We let $W_f$ be the Weyl group for $\bt \subset \k_f$
and $W_{M_f}$ be the Weyl group of $\bt \subset \m_f$. 
Note that, in general, the natural injection $W_f \hookrightarrow W_{M_f}$ may not be surjective.

We say, as usual, that a parabolic subgroup is standard 
if it contains $A=\exp(\mathfrak{a}_{0\R})$. 
Finally let $Q$ be another standard parabolic subgroup with Levi decomposition\footnote{
Note that neither $M_Q$ nor $M_f$ need be connected.  }
$Q = M_Q A_Q N_Q$.

\subsubsection{Bilinear forms.}
We equip $\mathfrak{t}_U$ with the bilinear form which is the negative of the Killing form for $U$;
we equip $\at$ and $\a$ with the bilinear form defined by the Killing form for $G$. 
These forms are all nondegenerate, and thereby equip the dual spaces $\mathfrak{t}_U^*, \at^*, \a^*$ with bilinear forms;
if we write, for instance, $\langle \alpha, \beta \rangle$ for some $\alpha, \beta \in \a^*$, 
it always means that this is taken with respect to the bilinear form just normalized. 

\subsubsection{Systems of positive roots} \label{roots}
We choose a system $\Delta^+(\bt, \m_f)$ of positive roots for the action of 
$\bt$ on $\m_f$; it restricts to a system of positive roots 
  for $\bt$ on $\k_f$. We extend it to a system $\Delta^+ = \Delta^+ (\mathfrak{t}_U, \mathfrak{u})$ of positive roots
  for the action of $\mathfrak{t}_U$ on $\mathfrak{u}$: a root $\alpha$ of $\mathfrak{t}_U$ on $\mathfrak{u}$
  is positive if and only if  either the corresponding root space 
  belongs to $\mathfrak{n}_f$, the complexified Lie algebra of $N_f$; or, otherwise,  
$\alpha$ is trivial on $\mathfrak{a}_0$ and $\alpha \in \Delta^+(\bt, \m_f)$. 
  
Let $\rho_U \in \mathfrak{t}_U^*$
be the half-sum of positive roots in $\Delta^+$. 
Define similarly $\rho_{M_f}, \rho_{K_f}$.

\subsubsection{Representations}
Let $\rho_{\lambda}$ be an irreducible  representation of $U$ with (dominant) highest weight $\lambda \in \mathfrak{t}_U^{*}$; it extends to a unique holomorphic representation of $G_{\C}$, which we also denote $\rho_{\lambda}$.    

Given $\theta$ a unitary representation of $M_Q$ and $\nu \in \mathfrak{a}_Q^* =  \Hom(\mathfrak{a}_Q, \C)$ we may construct a (generalized) principal series representation $\pi(\theta, \nu)$
of the group $G$; its restriction to $K$ is isomorphic to $\mathrm{Ind}^{K}_{(K \cap M_Q)} \theta$. 
We normalize so that $\pi(\theta,\nu)$ is tempered unitary if $\theta$ is tempered
and $\nu|\mathfrak{a}_{Q, \mathbb{R}}$ is valued in $i \mathbb{R}$. 
Note that this principal series is not necessarily irreducible;
however, for $\theta$ irreducible and generic $\nu$, it is indeed irreducible: \cite[Theorem 7.2]{Knapp}.

For each $w \in W_U$, let $\mu_w$ be the restriction of $w (\rho_U + \lambda)$ to $\mathfrak{b}$. 
For any $\nu \in \mathfrak{a}^*$, we regard $\mu_w + \nu$ as defining an element
of $(\mathfrak{a} \oplus \mathfrak{b})^* = \mathfrak{t}_U^*$, in the obvious way. Put 
$$P_w(\nu) = \sign(w)  \prod_{\alpha \in 
\Delta^+(\mathfrak{t}_U, \mathfrak{u}) } \frac{\langle \mu_w+ \nu, \alpha \rangle}{\langle \rho_U , \alpha \rangle}, $$
where $\langle -, - \rangle$ is induced by the Killing form, as described above. 
Then $P_w$ depends only on the coset of $w$ in $W_{M_f} \backslash W_U$. 
Indeed, 
for $w' \in W_{M_f}$, 
$  \prod_{\alpha} \langle \mu_{w'w} + \nu, \alpha \rangle =
  \prod_{\alpha} \langle \mu_{w} + \nu , w'^{-1} \alpha\rangle$
and the set $w'^{-1} \alpha$ coincides with the set of $\alpha$ 
{\em after} a number of sign changes, of total parity $\sign(w')$.

\begin{prop} \label{L2explicit}  
If $\delta (S) \neq 1$, then $t_S^{(2)} (\rho )    = 0$.  If $\delta (S)=1$:
For each $w \in W_U$, put $J_w = \{ \nu \in \Hom(\mathfrak{a}_{0, \R}, \R): |\nu| \leq 
\sqrt{| \rho_U +\lambda|^2- |\mu_w|^2}\}$; equip $\Hom(\mathfrak{a}_{0, \R}, \R)$
with the additive-invariant measure in which the interval $|\nu| \leq 1$ has measure $2$. 
Then there exists $c(S) > 0$, depending only on $S$, so that:
\begin{eqnarray*}
t_S^{(2)} (\rho )     =  (-1)^{\frac{\dim S-1}{2}}  c(S)  
\sum_{w \in  W_{M_f}\backslash W_U}  \int_{J_w}  P_w(\nu) d\nu.
\end{eqnarray*}
 Moreover, when $\delta(S) = 1$, 
$$(-1)^{\frac{\dim S-1}{2}} t_S^{(2)}(\rho) > 0.$$
\end{prop}

Note that $t_S^{(2)} (\rho)$ depends on the choice of a normalization for the (symmetric) Riemannian metric on $S$. If we had scaled the metric on $S$ by a constant $C$, then $t_S^{(2)} (\rho)$ would be scaled by
$C^{-\dim (S)}$, as is the volume form on $S$. 

The proof  takes up the rest of this section; we work out some explicit examples in \S \ref{sl3} and prove that $t_S^{(2)} (\rho )$ is nonzero if $\delta (S)=1$ in \S \ref{Cpositive}.  Let us explain the outline.  We are not going to give details of issues related to convergence. 
It follows from \S \ref{erhodiff} and  \eqref{L2analyticTorsion} 
that 
\begin{equation}\label{tv22} 
t_S^{(2)} (\rho )   =  \frac{1}{2}  \frac{d}{ds} \Big|_{s=0} \frac{1}{\Gamma(s)} \int_{0}^{\infty}  dt \int_{\pi \in \hat{G}} d(\pi) t^{s-1}  e^{-t(\Lambda_{\rho} - \Lambda_{\pi})} d\muP (\pi)
\end{equation}
where  $d(\pi) \in \mathbb{Z} $ is defined as
\begin{equation}\label{dpidef} d(\pi) := \sum_{k} (-1)^k k \dim \pi \otimes \wedge^k \mathfrak{p}^* \otimes \rho_{\lambda}|_K,  \end{equation} $d\muP$ is the {\em Plancherel measure} on the unitary dual $\hat{G}$, and $\Lambda_{\rho}$, $\Lambda_{\pi}$ denote the Casimir eigenvalues of $\rho$ and $\pi$.  It is often convenient to think of $d(\pi)$ as the 
dimension of the {\em virtual vector space}
$ \sum^* k  [\pi \otimes \wedge^k \mathfrak{p}^* \otimes \rho_{\lambda}|_K]^K$ (alternating sum),

The proof of Proposition \ref{L2explicit} follows from
the following results:
\begin{enumerate}
\item $d(\pi(\theta,\nu)) = 0$ unless $Q$ is associate to $P_f$.  We prove this in \S \ref{Vanishing}. 
This proves, by standard facts about Plancherel measure, that the integration in \eqref{tv22}
may be restricted to the case when $Q = P_f$, $\dim A_f =1$ and $\theta$ is a discrete-series
representation of $M= M_f$.  
\item In \S \ref{fundoscopic} -- \S \ref{explicit3}, we evaluate $d(\pi(\theta,\nu))$ explicitly when $Q=P_f$ is a fundamental parabolic subgroup and $\theta$ a discrete series; namely, 
\S \ref{fundoscopic} sets up general notations, \S \ref{explicit2} recalls character formulae,
and \S \ref{explicit3} carries out the computation of $d(\pi)$. 
\item The explicit Plancherel density when $Q = P_f$ and $\theta$ discrete is presented
in \eqref{planch}; this, together with  \S \eqref{an-integral}, completes the computation 
of the right-hand side of \eqref{tv22}. 
\end{enumerate}

\subsection{}\label{Vanishing}
Let $G$ be a group and 
 $V$ is a $G$-vector space. We denote by $\det[1-V]$ the virtual $G$-representation
 (that is to say, element of $K_0$ of the category of $G$-representations)
 defined by the alternating sum $\sum_i (-1)^i [\wedge^i V]$ of exterior powers. 
 This is multiplicative in an evident sense:
\begin{equation} \label{add1} 
\det[1 - V \oplus W] = \det[1-V] \otimes \det[1-W].
\end{equation}

We put $\detp[1-V] = \sum_{i} (-1)^i i \wedge^i V$. Then
\begin{equation} \label{add} 
\detp[1 - V \oplus W] = \detp[1-V] \otimes \det[1-W] \oplus \det[1-V] \otimes \detp[1-W].
\end{equation} 

Let us note that if $H$ is a group acting on $V$ such that every $g \in H$ has a fixed space of dimension $\geq 2$,
then $\detp [1-V]$ is trivial as a virtual $H$-representation; this follows since, 
by \eqref{add}, the trace of every $g \in H$ on $\detp[1-V]$ is then trivial. 

Note that $\detp [1- \mathfrak{p}^*] = \detp [1- \mathfrak{p}]$ as virtual $K$-representations so that, by Frobenius reciprocity and \eqref{dpidef}, 
\begin{equation} \label{desid}   d(\pi(\theta, \nu)) = 
\dim [\theta \otimes \detp [1- \mathfrak{p}] \otimes \rho_{\lambda}]^{(K \cap M_Q)}. \end{equation}
Let us suppose \eqref{desid} is nonzero. 
Each $g \in K \cap M_Q$ fixes, in its conjugation action, $\mathfrak{a}_Q$;
on the other hand, it also belongs to some conjugate of a maximal torus in $K$ and,
as such, fixes some $K$-conjugate of $\mathfrak{a}$.  There must exist some $g \in K \cap M_Q$
for which these two spaces coincide and are one-dimensional; otherwise $\detp[1-\mathfrak{p}^*]$ is virtually trivial.
Therefore, $\mathfrak{a}$ and $\mathfrak{a}_Q$ are $K$-conjugate and one-dimensional. 
It follows that  $Q$ is fundamental (i.e., it is associate to $P_f$) and $\dim A_f = 1$. 

In particular, $t_S^{(2)}(\rho) =0 $ unless $\delta(S) = 1$. 

\subsection{The case of the fundamental series.} \label{fundoscopic}

In the remainder of the section, $Q = P_f$.

It is a theorem of Harish-Chandra \cite[Corollary 14.60]{Knapp} that, in these circumstances, $\pi(\theta, \nu)$ is irreducible for all $\nu$, although we do not need this fact.  Since compact 
Cartan subgroups of $M_f$ are connected, we may moreover realize $\theta$ as an induction 
from the connected component $M_f^0$ of a discrete-series representation $\theta_0$ of 
$M_f^0$ (see \cite[\S 6.9 and
\S 8.7.1]{Wallach}). Again by Frobenius reciprocity we obtain
\begin{equation} \label{desid2}   
d(\pi(\theta, \nu)) =  \dim [\theta_0 \otimes \detp [1- \mathfrak{p}] \otimes \rho_{\lambda}]^{K_f^0}. 
\end{equation}

 We denote by $\mu \in \bt^*$ the the infinitesimal character, of the representation $\theta_0^*$ --
 the dual representation to $\theta$. (We index $\theta$ by the infinitesmal character of its {\em dual} to make the computation shorter at a later stage.)    Therefore, $\theta_0^*$ contains with multiplicity one the $K_f^0$- type
with highest weight $\mu + \rho_{M_f} - 2 \rho_{K_f}$ (``Blattner's formula'', \cite[Theorem 9.20]{Knapp}). 
There exists a positive constant
  $c_S'$ depending only on our choice of  a Haar measure on $G$ such that the Plancherel measure along the space of $\pi(\theta, \nu)$ is equal to:
\begin{equation} \label{planch} 
c_S' (-1)^{\frac12 \dim N} \prod_{\alpha \in \Delta^+} \frac{\langle \mu +  \nu, \alpha \rangle}{\langle \rho_U , \alpha \rangle} \cdot d\nu 
\end{equation}
(see \cite[Theorem 13.11]{Knapp}). This is a nonnegative even polynomial on $i \mathfrak{a}_{0, \R}$.

\subsection{Weyl denominators} \label{explicit2}
We continue with the notation of \S \ref{fundoscopic}.  Let $T_f^{\sim}, T_U^{\sim}$ denote, respectively, the universal covers of $T_f, T_U$.

Our choice of positive systems defines {\em Weyl denominators} $D_{M_f}: T_f^{\sim}\rightarrow \mathbb{C}, D_U: T_U^{\sim} \rightarrow \mathbb{C}$;
these are,  by definition, the determinant
of $1 - \mathrm{Ad}(t^{-1})$ acting on the sum $\mathfrak{m}_f^+$ resp. 
$\mathfrak{u}^+$ of all positive root spaces
on $\mathfrak{m}_f$ resp. $\mathfrak{u}$, multiplied by $\rho_{M_f}$ or $\rho_U$. 
In a similar way we obtain a  Weyl denominator $D_{K_f}$
of $T_f^{\sim}$, where we sum over all positive roots of $\mathfrak{k}_f$ only. Note that formally Weyl denominators are given by products $\prod_{\alpha >0} (e^{\alpha/2} - e^{-\alpha/2})$ 
over the corresponding positive roots.

Now $\mathfrak{u} / \mathfrak{t}_U = \mathfrak{u}^+ \oplus
\mathfrak{u}^-$ and \eqref{add1} implies that the character of  $\det [1-\mathfrak{u}/\mathfrak{t}_U]$
is equal to $\prod_{\alpha >0} (1-e^{\alpha}) \cdot \prod_{\alpha >0} (1-e^{-\alpha})$.
It follows that:
\begin{equation} \label{denomdet} D_U^2 = (-1)^{\dim \mathfrak{u}^+} \mbox{ character of } \det [1-\mathfrak{u}/\mathfrak{t}_U],
\end{equation}  
with a similar identity for $D_{M_f}^2$ and $D_{K_f}^2$.

We claim that, as characters of $T_f^{\sim}$,  
\begin{equation} \label{Weyl-denom} 
|D_{K_f}|^2 \cdot   (\mbox{character of }\detp[1-\mathfrak{p}]) = - D_U \overline{D_{M_f}}
\end{equation}
(In particular, the right hand side descends to $T_f$). 

Equivalently,  since the $T_f$-character of $\detp [1-\mathfrak{p}] $ is the negative of the character
of $\det [1 -\p/\at] $ and $D_{K_f}^2$ (resp. $D_{M_f}^2$) is real-valued and has the same sign
as $(-1)^{\dim \k_f^+}$ (resp. $(-1)^{\dim \m_f^+}$),  we need to establish
\begin{equation}\label{pe2} \det [1 -\p/\at ] = (-1)^{\dim \m_f^+ + \dim \k_f^+} \frac{D_U}{D_{M_f}} \cdot \frac{D_{M_f}^2}{D_{K_f}^2}.\end{equation}

The {\em signs} of both sides of this putative equality are the same. In fact, 
we claim that $(-1)^{\dim \m_f^+ + \dim \k_f^+}D_{M_f}^2/D_{K_f}^2$, $D_{U}/D_{M_f}$, and the character
of $\det[1-\p/\at]$ all take positive values on $T_f^{\sim}$:
\begin{itemize}

 \item[-]  
$D_{M_f}^2$ is real-valued and has the same sign
as $(-1)^{\dim \m_f^+}$; similarly for $D_{K_f}^2$. It follows that the quantity $(-1)^{\dim \m_f^+ + \dim \k_f^+} D_{M_f}^2/D_{K_f}^2$ is positive.

\item[-] 
The representation of the compact torus $T_f$ on $\mathfrak{u}^+/\mathfrak{m}_f^+$ 
has no fixed vectors (the centralizer of $\mathfrak{b}$ is $\mathfrak{b} \oplus \mathfrak{a}_0$)
and is the complexification of the real representation on  $\mathfrak{n}_{f\R}$.

This means that  the eigenvalues of $T_f$ on $\mathfrak{u}^+/\mathfrak{m}_f^+$ come in complex-conjugate pairs; it also shows that 
 $\rho_U/\rho_{M_f} > 0$, being  a real valued character on the connected group $T_f^{\sim}$. 
 
 Therefore $D_U/D_{M_f}  > 0$. 
 
\item[-] The character of $\det[1-\p/\at]$ is positive because,
again, the eigenvalues of $T_f$ on $\p/\at$ come in complex conjugate pairs,
this being the complexification of   
 $\mathfrak{p}_\R/\mathfrak{a}_0 \cap \mathfrak{p}_{\R}$. 

\end{itemize}
 
  Note this
 also implies that 
 \begin{equation} \label{congruence} 
\dim \mathfrak{m}_f^+ \equiv \dim \mathfrak{u}^+ \mbox{ modulo } 2. 
\end{equation}

Since the signs match, it suffices to check \eqref{pe2} after squaring both sides. 
In view of \eqref{denomdet} and \eqref{congruence}, we are reduced to checking the equality of $T_f^{\sim}$-characters:
$\det [1 -\p/\at ]^2 = \det [1- \mathfrak{u}/\mathfrak{t}_U] \det [1 - \mathfrak{m}_f / \mathfrak{b}] / \det [1 -
\mathfrak{k}_f/ \mathfrak{b} ]^2$ which in turn would follow from
\begin{equation} \label{eqRep}
2 [\p / \at ] = [\mathfrak{u} / \mathfrak{t}_U] + [\mathfrak{m}_f / \mathfrak{b}] - 2 [\mathfrak{k}_f/ \mathfrak{b} ],
\end{equation}
this identity being understood in the Grothendieck group of $T_f$-representations. 

Now $[\mathfrak{u} ] =
[\mathfrak{k}] + [\mathfrak{p}]$ as $T_f$-representations. Similarly 
$[\mathfrak{m}_f] = [\mathfrak{k} \cap \mathfrak{m}_f] + [\mathfrak{p} \cap \mathfrak{m}_f]$. So we can
write the right hand side of \eqref{eqRep} as:
\begin{eqnarray*}
 [\mathfrak{u} / \mathfrak{t}_U] + [\mathfrak{m}_f / \mathfrak{b}] - 2 [\mathfrak{k}_f / \mathfrak{b} ] & = &  [\mathfrak{u}] + [\mathfrak{m}_f] - 2 [\mathfrak{k}_f] + [ \mathfrak{b} ] - 
[\mathfrak{t}_U] \\
& = & [\mathfrak{k}] + [\mathfrak{p}] + [\mathfrak{k}_f] + [\mathfrak{p} \cap \mathfrak{m}_f]
- 2 [\mathfrak{k}_f] + [ \mathfrak{b} ] - 
[\mathfrak{t}_U] \\
& = & [\mathfrak{k}/ \mathfrak{k}_f] + ([\mathfrak{p}] + [\mathfrak{b}] - [\mathfrak{t}_U])
+ [\mathfrak{p} \cap \mathfrak{m}_f] \\
& = &  [\mathfrak{k}/ \mathfrak{k}_f] + [\mathfrak{p}/\at]+
[\mathfrak{p} \cap \mathfrak{m}_f] .
\end{eqnarray*}

Write $\mathfrak{p}_M = \mathfrak{p} \cap (\mathfrak{m}_f \oplus \at)$. 
If we show that $\mathfrak{p}/\mathfrak{p}_M$ and $\mathfrak{k}/ \mathfrak{k}_f$ are isomorphic as $T_f$-representations, we will be done. 

Choose nonzero
 $X_0 \in \at$ and consider the map $$\phi: Y \in \k/\k_f \mapsto [Y,X_0] \in \mathfrak{p}.$$ 
 This map is injective, because $\mathfrak{k}_f$ exhausts the centralizer of $X_0$ in $\mathfrak{k}$; moreover, for any $Z \in \mathfrak{p}_M$, 
 $$\langle [Y, X_0] , Z \rangle = - \langle Y, [X_0, Z] \rangle = 0,$$
 so that the image of $\phi$ in fact lands in $\mathfrak{p}_M^{\perp}$.  
The map $\phi$ thus induces a $T_f$-equivariant injective map from $\mathfrak{k}/ \mathfrak{k}_f$ to 
$\mathfrak{p}/\mathfrak{p}_M$. The difference between the dimensions of these spaces is the difference between the split ranks of $G$ and $MA$; we verify by direct inspection\footnote{For 
$\SL_3(\R)$ both ranks are $2$, for $\SO(p,q)$ both ranks are $\mathrm{min}(p,q)$.}
that this is zero in all cases. These spaces are thus equivariantly isomorphic, as desired. 

\subsection{Computation for the fundamental series.} \label{explicit3} 
We shall now compute \eqref{desid} in the case under consideration, when $Q=P_f$ is fundamental
and the contragredient representation $\theta_0^*$ as indexed previously by a dominant character $\mu \in \bt^*$.  

Let $[W_f \backslash W_U]$ be the set of $w \in W_U$ such that $\mu_w$ is dominant as a weight on $\mathfrak{b}$ (with respect to the roots of $\mathfrak{b}$ on $\mathfrak{k}_f$), i.e.:
$$[W_f \backslash W_U] = \{ w \in W_U: \langle \mu_w, \beta \rangle \geq 0 \mbox { for all } \beta \in \Delta^{+}(\mathfrak{b}, \mathfrak{k}_f) \}.$$
This is therefore a set of coset
representatives for $W_f$ in $W_U$: for every such $\beta$, 
  \begin{equation} \label{dom} \langle (\rho_U + \lambda) , \beta \rangle \geq \langle \rho_U, \beta \rangle > 0.\end{equation}
This shows that any $\mu_w$ lies strictly in the interior of a Weyl chamber for
the roots of $\mathfrak{b}$ on $\mathfrak{k}_f$.

Let us agree to denote by $\sum^*$ any summation over a Weyl group
that is weighted by the sign character. 

If we equip $T_f$ with Haar measure of total mass $1$, the Weyl integration formula
formally gives:
\begin{eqnarray*}
d(\pi (\theta , \nu)) & = & \int_{T_f} \mbox{character of } (\theta_0 \otimes \detp[1-\mathfrak{p}] \otimes
\rho_{\lambda} ) \frac{|D_{K_f}| ^2}{|W_{f}|} dt .
\end{eqnarray*}
{\it A priori} this is only formal; its validity is a nontrivial fact, and only follows from the truth of Blattner's conjecture \cite{Schmid}.

We carry out the computation using character formulae for discrete series. 
 
The character of $\rho_{\lambda}$ is given on the
maximal torus $T_U$ by  
$$\frac{ \sum^*_{w \in W_U} w (\rho_U + \lambda)   }{ D_{U}}.$$
Separately, the numerator and denominator only make sense on $T_U^{\sim}$, but the ratio makes sense on $T_U$.
On the other hand, the character of $\theta_0^*$
is a distribution whose restriction to $T_f$ is given by \cite[Theorem 12.7]{Knapp}:
$$(-1)^{\frac{1}{2} \dim(M_f/K_f )} \frac{\sum^*_{w \in W_f}   \ w\mu } {D_{M_f}}.$$ 
Here again the ratio makes sense on $T_f$. It thus follows from \eqref{Weyl-denom}
that:
$$
\begin{array}{l}
d(\pi (\theta , \nu)) \\
\ \ \ = \frac{(-1)^{\frac{1}{2} \dim(M_f/K_f)}}{|W_{f}|} \int_{T_f} \overline{\left( 
\frac{\sum^*_{w \in W_f}   \ w\mu } {D_{M_f}} \right)} \left( \frac{ \sum^*_{w \in W_U} w (\rho_U + \lambda )   }{ D_{U}} \right) \detp[1-\mathfrak{p}] |D_{K_f}| ^2 dt \\
\ \ \ = (-1)^{\frac{1}{2} \dim(M_f/K_f)  +1} \int_{T_f} \overline{\left( 
\sum \- {}^*_{w \in W_f}    \ w\mu  \right)} \left( \sum\- {}^*_{w \in [W_f \backslash W_U]}  \mu_w   \right) dt.
\end{array}
$$ 
 The equation \eqref{dom} implies that $w \mu = \mu_{w'} \ \ (w \in W_f, w' \in [W_f \backslash W_U])$ only if $w$ is trivial; 
we conclude that
$$d(\pi (\theta , \nu)) = \left\{
\begin{array}{ll}
(-1)^{\frac{1}{2} \dim(M_f/K_f)  +1}  {\rm sgn} (w)  & \mbox{ if } \mu = \mu_w \mbox{ for some }  w \in [W_f \backslash W_U] \\
0 & \mbox{ if not}.
\end{array} \right. $$

\subsection{End of the proof of Proposition \ref{L2explicit}.} \label{5.7}
The infinitesimal character of $\rho_{\lambda}$ is $\rho_U + \lambda$. It thus follows \cite[Proposition 8.22 and Lemma 12.28]{Knapp} that   
$$\Lambda_{\rho} - \Lambda_{\pi (\theta , \nu)}  = |\nu|^2 - |\mu |^2 + | \rho_U + \lambda |^2.$$

\medskip
\noindent
{\em Remark.} Note that from the definitions $| \rho_U + \lambda |^2 - |\mu_w|^2 $ is always nonnegative. 
It moreover follows from Parthasarathy's Dirac inequality \cite{Parthsarathy,BorelWallach,VoganZuckerman} that  $\rho_{\lambda}$ is 
strongly acyclic if and only if $| \rho_U + \lambda |^2 - |\mu_w|^2 $ is
positive for all $w$. 

\medskip

Recall that we assume $\delta (G)=1$. 
From this, \eqref{planch}, \eqref{tv22}, and the result of \S \ref{explicit3} we
get:
$$
\begin{array}{ll}
t_S^{(2)} (\rho)  = &  \frac{(-1)^{\frac12 \dim N_f +\frac{1}{2} \dim(M_f/K_f) +1 }}{2} c_S'  \\
& \ \ \ \times  \sum_{w \in [W_f \backslash W_U]} 
 {\rm sgn} (w)  \frac{d}{ds} \Big|_{s=0} \Big( \frac{1}{\Gamma(s)} \int_{0}^{\infty}  dt \\
& \ \ \   \int_{-\infty}^{+\infty} t^{s-1} e^{-t (\nu^2 + | \rho_U + \lambda |^2 - |\mu_w|^2) } \prod_{\alpha \in \Delta^+(\mathfrak{t}_U, \mathfrak{u})} \frac{\langle \mu_w +  \nu, \alpha \rangle}{\langle \rho_U , \alpha \rangle} \cdot d\nu \Big) .
\end{array}
$$
It moreover follows from \S \ref{an-integral} that if $p$ is an even polynomial and $c\geq 0$,
\begin{eqnarray*}
\frac{d}{ds} \Big|_{s=0} \frac{1}{\Gamma(s)} \int_{0}^{\infty}  dt  \int_{-\infty}^{+\infty} t^{s-1} e^{-t (y^2+ c^2)} p(iy) d y = - 2\pi \int_0^c p(y) dy.
\end{eqnarray*}
To conclude note that $\dim N_f + \dim (M_f / K_f) = \dim S - 1$.  Thus we obtain the statement of Proposition \ref{L2explicit}, where the constant $c(S)$ is given by $ \frac{\pi}{2} \frac{|W_{M_f}|}{|W_{K_f}|} c_S'$, and $c_S'$
is the constant appearing in the Plancherel measure \eqref{planch}.
\qed

\subsection{Non-vanishing in the general case} \label{Cpositive}
Assume $\delta (S) =1$; we shall show that $t_S^{(2)}(\rho)$ is always nonzero
and of the sign $(-1)^{\frac{\dim S-1}{2}}$. We follow an idea of Olbrich to minimize computation. 

\proof 
 
If $\delta(S) = 1$, the (real) Lie algebra of $G$ splits as $\mathfrak{g}_0 \oplus \mathfrak{g}_1$, where
$\delta(\mathfrak{g}_0) = 0$ and $\mathfrak{g}_1 $ is isomorphic to either $\mathfrak{sl}_3$
or $\mathfrak{so}_{p,q}$. 
Using Proposition \ref{L2explicit}
we reduce to checking in the case where $\mathfrak{g}_0$ is trivial, i.e. 
the (real) Lie algebra of $G$ equals either $\mathfrak{sl}_3$ or $\mathfrak{so}_{p,q}$.

Moreover, the expression $\sum_{w} \int_{\nu}$ of Proposition \ref{L2explicit} depends only on the triple
$(\mathfrak{a}_0 \subset \mathfrak{t}_U \subset \mathfrak{u})$.  The isomorphism class of this triple, in the case of $\mathfrak{so}_{p,q}$,
depends only on $p+q$ (more generally, it depends only on the inner form of the Lie algebra). 

We are reduced to checking the cases when the real Lie algebra of $G$ 
is isomorphic to either $\mathfrak{sl}_3$ or $\mathfrak{so}_{p,1}$. 
These cases are handled in \S \ref{SL3} and \S \ref{Ghyp} respectively. 
\qed

\medskip
\noindent
{\em Proof of Theorem \ref{main}.}
The existence of strongly acyclic bundles is proven in \S \ref{SAexist}. 

It follows from (\ref{TX=tor}), Theorem \ref{approxthm}, and duality that
$$\sum \- {}^* \log |H_j (\Gamma_N , M)|
\rightarrow {\rm vol} (\Gamma \backslash S) t_S^{(2)} (\rho )$$ 
as $N$ tends to infinity. 
The constant $c_{G, M} = t_S^{(2)} (\rho )$ has the same sign as $(-1)^{\frac{\dim S-1}{2}}$ as we have just proven, and only depends
on $G$, $M$. 
\qed

\subsection{Examples.}  \label{sl3}

\subsubsection{$\mathbf{G} = \SO_{2n+1,1}$} \label{Ghyp}

In this case $U \cong \SO_{2n+2} (\R) $, $K \cong \mathrm{O}_{2n+1}(\R)$, $M_f=K_f \cong \SO_{2n}(\R)$ and
$S= {\Bbb H}^{2n+1}$. %

In the notation of \cite{Bourbaki}, we may choose a Killing-orthogonal basis $\varepsilon_i$ for $\mathfrak{t}_U^*$
such that:
\begin{enumerate}
\item 
The positive roots are those roots $\varepsilon_k \pm \varepsilon_l$ with $1 \leq k < l \leq n+1$;
\item $\mathfrak{a}$ is the common kernel of $\varepsilon_2, \dots, \varepsilon_n$;
\item $\mathfrak{b}$ is the kernel of $\varepsilon_1$, and the
positive roots for $\mathfrak{b}$ on $M_f \cap K$ are $\varepsilon_j \pm \varepsilon_k \ (1<  j < k)$.
\item The positive roots nonvanishing on $\mathfrak{a}$ are $\varepsilon_1 \pm 
\varepsilon_l$ ($1 < l \leq n+1$), and thus $\alpha_0 = \varepsilon_1$ gives
the unique positive restricted root $\mathfrak{a}_{0, \R} \rightarrow \mathbb{R}$. 
 \end{enumerate}

The representations of $U$ are parametrized by a highest weight 
$\lambda = (\lambda_1 , \ldots , \lambda_{n+1}) = \lambda_1 \varepsilon_1 + \ldots + \lambda_{n+1} \varepsilon_{n+1}$ such that $\lambda$ is dominant (i.e.
$\lambda_1 \geq \cdots \geq \lambda_n \geq |\lambda_{n+1}|$) and integral (i.e. every $\lambda_i \in \Z$).  The action of the Cartan involution on representations is via:
$$(\lambda_1, \dots, \lambda_{n+1}) \mapsto (\lambda_1, \dots , \lambda_n, -\lambda_{n+1}).$$
Note in particular (see Lemma \ref{SA}) that $\rho= \rho_{\lambda}$ is strongly acyclic if and only if $ \lambda_{n+1} \neq 0$.

The Weyl group of $U$ consists of
permutations and sign changes of $\{ \varepsilon_1 , \ldots , 
\varepsilon_{n+1} \}$ with even sign (i.e., positive determinant). 
  In particular 
$|W_U| = 2^{n} (n+1) !$.  The subgroup $W_f = W_{M_f} \subset W_U$ comprises all permutations
that fix the first coordinate. 

We have $\rho_U + \lambda = (n + \lambda_1 , n-1 + \lambda_2 , \ldots , \lambda_{n+1}).$
We may assume that $\lambda_{n+1} \geq 0$ and for convenience we will rewrite  
$$ \rho_U + \lambda = (a_{n} , \ldots , a_0).$$
Note that $(a_j)$ is a strictly increasing sequence of nonnegative integers.

The set $W_{M_f} \backslash W_U$ has size $2(n+1)$;  
one may choose a set of representatives for it given by
$w_k \ (0 \leq k \leq n)$, any element sending $\varepsilon_k$ to $\varepsilon_1$,
as well as $w_k^*$, any element sending $\varepsilon_k$ to $-\varepsilon_1$.  
For explicitness we choose $w_k$ so that
\begin{eqnarray} \label{weighta}
 w_k (\rho_U + \lambda) = 
(  a_{k} , a_n , \ldots , \widehat{a_k} , \ldots , a_0);
\end{eqnarray}
it is of sign $(-1)^{n-k}$. (As usual, $\widehat{a_k}$ denotes that the $a_k$ term is {\em omitted.})    
Choose $w_k^*$ similarly to $w_k$ but replacing $a_k, a_0$ by $-a_k, -a_0$ respectively. 
Thus
$\mu_{w_k} + t \alpha_0 = (t, a_n, \ldots, \widehat{a_k}, \ldots,  a_0)$.

To apply Proposition \ref{L2explicit}, first note
$|\rho_U + \lambda|^2 - |\mu_{w_k}|^2 = a_k^2 |\varepsilon|^2$, where $|\varepsilon|^2$ is the common value of any $|\varepsilon_j|^2$. 
This again shows that $\rho_{\lambda}$ is strongly acyclic when $a_0 = \lambda_{n+1} \neq 0$. 
Set $$E=E(\rho )  =
  \prod_{0 \leq i < j \leq n} (a_j^2 - a_i^2) , \ \ F  = \prod_{0 \leq i < j \leq n} (j^2 -i^2) .$$
When $\rho$ is trivial, $ E = F$. 
 Now, notation as in Proposition \ref{L2explicit}, 
\begin{eqnarray} \label{comp2}
P_{w_k}(t \alpha_0) =  P_{w_k^*}(t \alpha_0)   & = & \frac{E}{F} \prod_{j \neq k} \frac{t^2 - a_j^2}{a_k^2 - a_j^2}.
\end{eqnarray}
Note that the left-hand side is independent of the normalization of the inner product $\langle \cdot, \cdot \rangle$. 
The polynomial $\Pi_k (t) = \prod_{j \neq k} \frac{t^2 - a_j^2}{a_k^2 - a_j^2}$ is 
of degree $2n$ and $\Pi_k ( \pm a_j) = \delta_{jk}$, the Kronecker $\delta$ symbol. Set
$Q_k = \Pi_k + \ldots + \Pi_n$; it is the unique even polynomial of degree $\leq 2n$ which satisfies
$$Q_k (\pm a_j) = \left\{
\begin{array}{ll}
0 & \mbox{ if } j <k \\
1 & \mbox{ if } j \geq k.
\end{array} \right.$$

  Proposition \ref{L2explicit} 
then implies that: 
\begin{eqnarray} \label{Thyp}
\nonumber t_{{\Bbb H}^{2n+1}}^{(2)} (\rho ) & = & (-1)^n c(\mathbb{H}^{2n+1})  \frac{4 |\varepsilon| E}{F}
\sum_{k=0}^{n} \int_0^{a_k} \Pi_k (t) dt. \\ 
& = & (-1)^{n} c(\mathbb{H}^{2n+1}) \frac{4 |\varepsilon| E}{F} \sum_{k=0}^n \int_{a_{k-1}}^{a_k} Q_k (\nu) d\nu.
\end{eqnarray}
Here we set $a_{-1} = 0$.

Each integral in \eqref{Thyp} is positive. In fact $Q_k '$ has a root in each interval $[\pm a_{j-1} ,  \pm a_j]$ for $1\leq j \neq k$, as well as a root in $[-a_0 ; a_0]$. Being of degree $\leq 2n-1$ this forces $Q_k$
to be either constant (this is the case iff $k=0$) or strictly increasing between $a_{k-1}$ and $a_k$. It follows that $\int_{a_{k-1}}^{a_k} Q_k (\nu) d\nu >0$.

\subsubsection{$\mathbf{G}=\SL_3$.} \label{SL3}

In this case $U \cong \SU_3$, $K \cong \SO_3$, $M_f \cong \{g \in \GL_2(\R): \det g = \pm 1\}$,   and $K_f \cong \mathrm{O}_2(\R)$. 

Fix an element $g \in \mathrm{SU}_3$ conjugating $\mathfrak{t}_U$ into diagonal matrices; 
let $\varepsilon_j \ \ (1 \leq j \leq 3)$ be the pull-back, by $g$, of the coordinate functionals. 
Thus $\sum \varepsilon_j = 0$; moreover,   we may choose $g$ such that:
\begin{enumerate}
\item The positive roots are $\varepsilon_i - \varepsilon_j$, with $i < j$; 
\item $\rho_U = \varepsilon_1 - \varepsilon_3$; 
\item $\mathfrak{b}$ is identified with the kernel of $\varepsilon_1 + \varepsilon_2 - 2 \varepsilon_3$; 
\item $\alpha_0 =  \frac{1}{2} (\varepsilon_1 + \varepsilon_2 - 2 \varepsilon_3)$ 
is the unique restricted root $\mathfrak{a}_{0,\R} \rightarrow \mathbb{R}$; 
\item The Weyl group $S_3$ acts by permuting the $\varepsilon_j$.
\end{enumerate}

Now let $p \geq q \geq r \in \Z$ and set $\lambda = p \varepsilon_1 + q\varepsilon_2 + r \varepsilon_3$.    Put
$$A_1 =\frac12 (p+1-q), \ \ \ A_2 = \frac12 (p-r+2), \ \ \ \mbox{ and } \ \ \ A_3= \frac12 (q-r+1).$$
$$C_1 = \frac13 (p+q-2r+3), \ \ \ C_2 = \frac13 (p+r-2q), \ \ \ \mbox{ and} \ \ \ C_3 = \frac13 (2p -q-r+3).$$
  Proposition \ref{L2explicit} gives, after a routine computation, 
$$t_S^{(2)}(\rho) =  2 | \alpha_0| c(S) \sum_{k=1}^3 (-1)^{k+1} \int_{0}^{|C_k|}  A_k (\left( \frac{9t^2}{4} \right) - A_k^2)dt.$$ 
Now $A_j \int_0^{|C_j|} (\frac{9 t^2}{4} - A_j^2 ) dy  = \frac{A_j |C_j |}{4} \left( 3 C_j^2 - 4 A_j^2 \right)$.    The inner $k$-sum equals
$$ 8 A_1 A_3 C_1 C_3  + 8 A_2 |C_2|  \begin{cases}  A_3 C_3,  \ \ \ C_2 \geq 0 \\
A_1 C_1,   \ \ \  C_2 \leq 0,
\end{cases}$$
which is manifestly positive, since $A_1, A_2, A_3, C_1, C_3$ are all positive.

\subsubsection{Numerical computations} \label{table} 
Here we compute the explicit value of the constant $c(S)$ appearing in Proposition \ref{L2explicit}. The
computations are not original, they were first done by Olbrich \cite{Olbrich} following Harish-Chandra.
We just translate these in our setting. 

First
recall from \S \ref{5.7} that $c(S) = \frac{\pi}{2} \frac{|W_{M_f}|}{|W_{K_f}|} c_S '$, where $c_S'$ is the constant appearing in the Plancherel measure \eqref{planch}. This constant is explicitely computed
by Harish-Chandra \cite[\S 24 Thm. 1 and \S 27 Thm. 3]{HC2} but we have to take into account that
Harish-Chandra's and our normalizations of the measure $dg$ and $d\nu$ differ by some multiplicative
factor. 

Indeed the Plancherel measure \eqref{planch} depends on the normalization of the Haar measure $dg$
and of $d\nu$. Our normalization (cf. \S 3.3) of $dg$ differ from Harish-Chandra's by the factor $2^{\frac12 (\dim (S) - \dim \mathfrak{a}_{\R})}$ -- see \cite[\S 37 Lemma 2]{HC1}. Similarly our $d\nu$ is 
the Lebesgue measure which corresponds to the metric induced by the Killing form on ${\rm Hom} (\mathfrak{a}_{0,\R} , 
\R )$. It differs from Harish-Chandra's by the factor $(2\pi)^{\delta(G)}$.  

We thus get 
\begin{eqnarray} \label{cS'}
c_S ' = \frac{1}{|W_A| (2\pi)^{\frac12 (\dim G/K + \delta (G))}} \frac{\prod_{\alpha \in \Delta^+} \langle \rho_U , \alpha \rangle}{\prod_{\alpha \in \Delta_{\mathfrak{k}}^+} \langle \rho_K , \alpha \rangle},
\end{eqnarray}
where $W_A = \{ k \in K \; : \; \Ad (k) \mathfrak{a} \subset \mathfrak{a} \} / K_f$, $\Delta_{\mathfrak{k}}^+ = \Delta^+ (\mathfrak{b} , \mathfrak{k})$ and $\rho_K \in \mathfrak{b}^*$
is the half-sum of positive roots in $\Delta_{\mathfrak{k}}^+$.

But \cite[\S 37 Lemma 4]{HC1} implies that 
$$\frac{\prod_{\alpha \in \Delta^+} \langle \rho_U , \alpha \rangle}{\prod_{\alpha \in \Delta_{\mathfrak{k}}^+} \langle \rho_K , \alpha \rangle} = (2\pi)^{\frac12 (\dim (U/T_U) - \dim (K / B))} \frac{{\rm vol} (K) {\rm vol} (T_U)}{{\rm vol} (U) {\rm vol} (B)},$$
where the volumes are the Riemannian ones corresponding to the Killing form. 

We may now suppose that $\delta (G)=1$ and let $\alpha_0$ the unique positive restricted root 
$\mathfrak{a}_{0,\R} \rightarrow \R$. Then ${\rm vol} (T_U) / {\rm vol} (B) = 2\pi / |\alpha_0|$ and we get:
\begin{eqnarray} \label{cS'2}
\frac{\prod_{\alpha \in \Delta^+} \langle \rho_U , \alpha \rangle}{\prod_{\alpha \in \Delta_{\mathfrak{k}}^+} \langle \rho_K , \alpha \rangle} = (2\pi)^{\frac12 (\dim (G/K) - 1)} \frac{2\pi}{|\alpha_0| {\rm vol} (S^c)}.
\end{eqnarray}

We finally  verify by direct inspection that $\frac{|W_{M_f}|}{|W_{K_f}| |W_A|}=2$ except for $G= \SO_{2n+1 , 1}^0$ where $\frac{|W_{M_f}|}{|W_{K_f}| |W_A|}=\frac12$. It thus follows from \eqref{cS'} and \eqref{cS'2} that:
\begin{eqnarray} \label{cS}
c(S) = \left\{ 
\begin{array}{ll}
\frac{\pi}{|\alpha_0| {\rm vol} (S^c)} & \mbox{ if } G \neq  \SO_{2n+1 , 1}^0 , \\
\frac{\pi}{4 |\alpha_0| {\rm vol} (S^c)} & \mbox{ if } G =  \SO_{2n+1 , 1}^0.
\end{array} \right.
\end{eqnarray}

\medskip
\noindent
{\it Examples.} 1) Since the volume of the standard $(2n+1)$-sphere is $2\pi^{n+1} /n!$ we get
$c(\H^{2n+1}) =  \frac{n!}{8 |\alpha_0| \pi^n}$ so that:
$$t_{{\Bbb H}^{2n+1}}^{(2)} (\rho )  =  (-1)^n \frac{n!}{2\pi^n} \frac{|\varepsilon| E}{F}
\sum_{k=0}^{n} \int_0^{a_k} \Pi_k (t) dt.$$
(Notations are as in \S \ref{Ghyp}.)

When $\rho$ is the trivial representation we recover that $t_{{\Bbb H}^{3}}^{(2)} = -\frac{1}{6\pi}$, 
$t_{{\Bbb H}^{5}}^{(2)} = \frac{31}{45\pi^2}$, ... see \cite{Luck} for more values. 

2) There is a natural Haar measure $dg$ on $\SL_3 (\R)$ which gives covolume $\zeta (2)
\zeta (3)$ to $\SL_3 (\Z)$. As in \S 3.3 the measure $dg$ is the product of a Riemannian volume form $dx$ on 
$S=\SL_3 (\R) / \SO_3 (\R)$ and the Haar measure $dk$ on $K$ normalized so that 
${\rm vol} (K) = {\rm vol} ({\Bbb S}^1) {\rm vol} ({\Bbb S}^2)= 8\pi^2$. The volume form $dx$ 
is associated to the invariant metric induced from the trace
form of the standard representation of $\mathfrak{sl}_3 (\R)$. Then a routine computation -- using e.g. 
Macdonald's formula \cite{Macdonald} -- shows that ${\rm vol} (S^c) = \pi^3 / \sqrt{2}$ and we deduce
from \S 5.9.2 that
$$t_S^{(2)} ({\rm triv}) = \frac{\sqrt{2}}{\pi^2}$$
so that 
$$t_S^{(2)} ({\rm triv}) \times {\rm vol} ( \SL_3 (\Z) \backslash S) = \frac{\sqrt{2}}{48\pi^2} \zeta (3) \approx 0,003.$$

3) One of our most important example is  $G= \SL_2 (\C)$. It is locally isomorphic to $\SO_{3,1}$ but
we have $U \cong \SU_2 \times \SU_2 $ and representations of $U$ are more easily parameterized by two non-negative integers $(p,q)$;  namely,
$$g_{\theta,\phi} =  \left( \begin{array}{cc} e^{i \theta} & 0 \\0 & e^{-i \theta} \end{array} \right)  
\times \left( \begin{array}{cc} e^{i \phi} & 0 \\0 & e^
{-i \phi} \end{array} \right) \mapsto
e^{i p \theta + i q \phi}.$$
The action of the Cartan involution on representations corresponds to the 
involution $\alpha \leftrightarrow \beta$ on highest weights.
The representation $\rho_{\lambda}$ is strongly acyclic if and only if $p \neq q$. The corresponding
finite dimensional representation of $\SO_{3,1}^0$ is parametrized by $(a_1 , a_0) = (p+q+1 , p-q)$ 
and we find:
\begin{eqnarray*}
t_{{\Bbb H}^{3}}^{(2)} (\rho ) =  \frac{-1}{6\pi} \left\{ (p+q+1)^3 - |p-q|^3 +3 (p+q+1) |p-q| (p+q+1- |p-q|) \right\} . 
\end{eqnarray*}

\medskip

One interesting possibility is that one could {\em numerically compute} $L^2$-torsion
sufficiently rapidly to be of interest as an algorithm for computing homological torsion. 
We do not know how efficient one could expect this to be, in practice.

\section{Ash-Serre conjectures and Conjecture \ref{conjmain}.} \label{Ash}

We shall
discuss the Ash-Serre conjectures, relating torsion classes for arithmetic groups to Galois representations; after this, we explain why these conjectures, together with various
known heuristics in number theory, are ``compatible'' with Conjecture \ref{conjmain}.  In particular,
we shall see how the term $\rank(G) - \rank(K)$ in Conjecture \ref{conjmain}
arises naturally from arithmetic computations.

This section, by contrast with the previous, is very algebraic and should also be regarded as very speculative in character.  We do prove, however, two results that are possibly of independent interest,
Proposition \ref{odd} (classification of odd involutions in the $L$-group) and Proposition \ref{MF} (counting the number of local nearly ordinary representations.)

\begin{prop}\label{odd}Let $\mathbf{G}$ be a semisimple simply connected algebraic group over the real numbers.  
Let $\LG$ be the Langlands dual group. Every involution in $\LG$ 
lifting the nontrivial element of $\mathrm{Gal}(\C/\R)$ has trace $\geq \rank(G) - 2 \ \rank(K)$ in the adjoint representation;
the set of elements of $\LG$ with adjoint trace equal to $\rank(G) - 2 \ \rank(K)$ forms a single $G^{\vee}(\C)$-orbit. 
 \end{prop}

This proposition has also been independently proven by M. Emerton and will appear in a published version of  \cite{CE}. 

This motivates the definition of {\em odd} Galois representation that will play a critical role in our later discussion (\S \ref{odddefn}): a Galois representation with target $\LG$ is {\em odd} if 
the image of complex
conjugation lies in the canonical class of involutions specified by the Proposition. 
 
\proof  It suffices to verify the assertion of the Proposition in the case when the Lie algebra $\mathfrak{g}$ 
 of $\mathbf{G}$ is $\R$-simple and quasisplit (the latter because the $L$-group 
 is unchanged by inner twisting). 

We shall use the fact that involutions correspond to real forms of semisimple Lie algebras together with a case-by-case check.  
More precisely, if $\mathfrak{h}$ is any complex semisimple Lie algebra, there
is a bijection:
\begin{equation} \label{basic-id} \mbox{involutions of $\mathfrak{h}$ modulo inner automorphism} \leftrightarrow \mbox{real forms of $\mathfrak{h}$.} \end{equation}
The map from the right-hand to the left-hand side is given by the (conjugacy class of) Cartan involutions associated to a real form. 
There
are natural equivalence relations on both sides: on the left, we 
have the fibers of the projection
to $\mathrm{Out}(\mathfrak{h})$; on the right, 
we have the equivalence relation of ``inner twisting.'' These relations are identified under \eqref{basic-id}.

Let us recall, moreover, that any real semisimple Lie algebra $\mathfrak{h}_0$ possesses 
a unique {\em quasi-split inner form} $\mathfrak{h}_0^{qs}$. It is numerically distinguished
amongst inner forms of $\mathfrak{h}_0$ by the fact that the Cartan involution has minimal trace.\footnote{This follows from the $NAK$ decomposition; since the Lie algebra of $NA$
is solvable, its complexification has complex dimension at most the dimension of a complex Borel, with equality if and only if the group is quasi-split.}

Together with \eqref{basic-id}, this implies that there is a unique conjugacy class 
of involutions in $\LG$ of minimal trace on $\mathfrak{g}$, 
but does not identify this trace. Let $\mathfrak{g}^{\vee}_0$
 be the quasi-split real form of $\mathfrak{g}^{\vee}$ 
 whose associated involution (under the bijection of \eqref{basic-id})  projects, in $\mathrm{Out}(\mathfrak{g}^{\vee})$, to the image of the complex conjugation under the map
 $\mathrm{Gal}(\C/\R) \rightarrow   \mathrm{Out}(\mathfrak{g} ) = \mathrm{Out}(\mathfrak{g}^{\vee})$. 

 Our assertion now amounts to the following: if $\Theta$ is the involution of $\mathfrak{g}^{\vee}$ obtained by complexifying the Cartan involution of $\mathfrak{g}_0^{\vee}$, then \begin{equation} \label{newid}\mathrm{trace}(\Theta) = \rank(G) - 2 \ \rank(K). \end{equation}

 We verify this case-by-case. 
 The map $\mathfrak{g} \rightarrow \mathfrak{g}^{\vee}_0$ is an involution 
 on isomorphism classes of simple real quasisplit Lie algebras; it fixes:
 \begin{enumerate}
\item The split forms of  $\mathfrak{g}_2, \mathfrak{f}_4, \mathfrak{e}_7, \mathfrak{e}_8$ and $\mathfrak{so}(2n, 2n)$. Both sides
 of \eqref{newid} equal $-\rank(G)$; 
\item $\mathfrak{so}(2n-1, 2n+1)$; both sides of \eqref{newid} equal $2-2n$. 
\item Every $\mathfrak{g}$ that is not absolutely simple, i.e., any complex simple Lie algebra, considered as a real Lie algebra; in those cases, both sides of \eqref{newid} are $0$. 
\end{enumerate} 
 As for the remaining cases, the involution switches 
 pairs of rows in the table:
  \begin{equation}
 \begin{array}{|c|c|c|}
 \hline
 \mathfrak{g} & r_G- 2 r_K  & \mathrm{trace}(\Theta) \\ 
 \hline
 \mathfrak{su}(p,p) & -2p & 0 \\
 \mathfrak{sl}_{2p} & 0 & - 2p \\
 \hline
\mathfrak{su}(p,p+1) & -2p-1 & 1 \\
\mathfrak{sl}_{2p+1} & 1 & -2p-1 \\
\hline
\mathfrak{so}(p,p+1) & -p & -p \\ 
\mathfrak{sp}(2p) & -p & -p \\
\hline
\mathfrak{so}(2p+1,2p+1) & 1-2p & -1-2p \\
\mathfrak{so}(2p,2p+2) & -1-2p & 1-2p \\
\hline
\mathfrak{e}_6^{\mathrm{split}} & -2 &  -6 \\
\mathfrak{e}_6^{\mathrm{qs}} & -6 & -2 \\ 
\hline
 \end{array}
 \end{equation}
 In the last line, $\mathfrak{e}_6^{\mathrm{qs}}$ denotes the unique {\em quasi-split, non-split}
 form of $\mathfrak{e}_6$. 
 \qed

\subsubsection{Examples.}
\begin{enumerate}
\item For $\mathbf{G}=\SL_n$, an odd involution on $\mathfrak{g}^{\vee} = \mathfrak{sl}_n$
is conjugation by:
$$\pm  \left(
\begin{array}{ccccc}
1& & & & \\
& -1 & & & \\
& & 1 & & \\
& & & -1 & \\
& & & & \ddots
\end{array} \right).$$
\item For $\mathbf{G} = \mathrm{SU}_n$ , an odd involution on $\mathfrak{g}^{\vee} = \mathfrak{sl}_n$
is represented by $X \mapsto -X^t$ (negative transpose). 
\item If $\mathbf{G}$ admits discrete series (i.e., $\mathrm{rank}(G)  = \mathrm{rank}(K)$)
then the {\em odd} involution is the unique conjugacy class with trace $- \mathrm{rank}(G)$.
This case is discussed in detail by B. Gross \cite{Gross}. 

\item If $\mathbf{G}_1$ is a complex reductive group, and $\mathbf{G}$ the Weil restriction of scalars to $\mathbf{R}$, then an odd involution $\mathfrak{g}^{\vee} = \mathfrak{g}_1^{\vee} \oplus \mathfrak{g}_1^{\vee}$
is given by $(X,Y) \mapsto (Y,X)$. 
\end{enumerate}

\subsection{Modularity conjecture.}  \label{localparameters}
We formulate a slight generalization of \cite{Ash} (split semisimple group, as opposed to $\mathrm{SL}_n$); we don't address the difficult question
of pinning down {\em all} the weights in which a given Galois representation occurs (see e.g  
\cite{Emerton, Herzig, bdj} and an unpublished manuscript by Clozel and Belabas where they
make precise computations for $\SL_2$ over an imaginary quadratic field).

\subsubsection{}
Let $\mathbf{G}$ be a semisimple  simply connected {\em split}\footnote{The assumption of {\em split} is not necessary; we have
imposed it to simplify the notation at several points.   The main goal of this section is to understand the arithmetic significance
of the factor $\rank(G) - \rank(K)$, and this is already interesting in the split case. } algebraic group over $\Z$. Let $G^{\vee}$ be the dual group. We regard it as a split group over $\Z$.  

For later use, we fix a maximal torus $\mathbf{T} \subset \mathbf{G}$,
a Borel $\mathbf{B} \subset \mathbf{G}$, and
thus a system of positive roots for $\mathbf{T}$.  Because $G$ is simply connected, the half-sum of positive roots gives a character
$\rho: \mathbf{T} \rightarrow \mathbb{G}_m
$.  

We also fix a torus $T^{\vee} \subset G^{\vee}$ dual to $\mathbf{T}$, and
a Borel $ B^{\vee} =T^{\vee} N^{\vee} \subset G^{\vee}$ containing $T^{\vee}$.

\subsubsection{}

Define the level $Q$ subgroup $\Gamma(Q) =  \mathrm{ker} \left( \mathbf{G}(\Z)  \rightarrow \mathbf{G}(\Z/Q\Z)\right).$ Let $\Gamma_0(Q)$ be the preimage of $\mathbf{B}(\Z/Q\Z)$
under the same map.

\subsubsection{}

  We define a {\em weight} to be an algebraic character $\bx$  
  of $\mathbf{T}$ such that $\bx - \rho$ is dominant. 
  We say that it is {\em strongly regular} if $\langle \bx, \alpha^{\vee} \rangle \geq 2$
  for every simple coroot $\alpha^{\vee}$.   In what follows, we always suppose that $\bx$ is strongly regular.

\subsubsection{} 

We shall extract from $\bx$ an arithmetic $\Gamma$-module $M_{\bx}$:

Being dominant, $\mathbf{x} - \rho$ is the highest weight of an irreducible representation 
$\alpha: \G_{\Q} \rightarrow \GL(W),$
where $W$ is a finite dimensional $\Q$-vector space. 
We choose a $\Gamma(1)$-stable $\Z$-lattice
$M_{\bx} \subset W$.  Such a lattice is not, in general, unique.

\subsubsection{} \label{odddefn}
 
Fix an algebraic closure $\Qbar$ of $\Q$ and a place of $\Qbar$
above every place $v$ of $\Q$. 
This gives an algebraic closure $\overline{\Q_v}$ of $\Q_v$ and an embedding
$D_v := \Gal(\Qvbar/\Q_v) \hookrightarrow \Gal(\Qbar/\Q)$.  We henceforth 
regard $D_v$ as a subgroup of $\Gal(\Qbar/\Q)$. Let $I_v \subset D_v$
be the inertia group for $v$ finite. For $v$ infinite, let $c_{\infty}$
be the nontrivial element of $D_v$.

\subsubsection{}

Let $k$ be a finite field of characteristic $p$ and of cardinality $q$.

Let $\omega: I_p \subset D_p \rightarrow \mathbb{F}_p^{\times} \hookrightarrow  \mathbb{G}_m(k)$ be the cyclotomic character (i.e., the map $D_p \rightarrow \mathbb{F}_p^{\times}$ arises
from the action on $p$th roots of unity). 
Define the {\em character associated to $\bx$}:
$$\psix :   I_p   \rightarrow T^{\vee}(k)$$
as the composition of $\omega$ with the dual $\bx^{\vee} :\mathbb{G}_m \rightarrow T^{\vee}$.

\subsubsection{}
A representation $\sigma_v : D_v \rightarrow G^{\vee}(k)$ is said to be:

\begin{enumerate}
\item
($v=\infty$) {\em Odd}, if $\sigma(c_{\infty})$ 
has trace  $\mathrm{rank}(G) - 2 \  \mathrm{rank}(K)$  in the adjoint action on $\mathfrak{g}$.

\item  ($v=p$) {\em Nearly ordinary of type $\bx$,} if it is conjugate
to a representation with image in $B^{\vee}(k)$, which, 
upon restriction to inertia and
projection
to $B^{\vee}/N^{\vee}(k) \cong T^{\vee}(k)$, equals $\psi_{\bx}$. 

  We will sometimes also apply the description {\em nearly ordinary at $\bx$}
to a representation with codomain $B^{\vee}(k)$; this means simply that
the projection to $B^{\vee}/N^{\vee}$ coincides, on inertia, with $\psi_{\bx}$. 

\item ($v=\ell \neq p$)
  {\em Steinberg at $\ell$,} 
if its restriction to $I_{\ell}$ has as image the cyclic subgroup generated
by a principal unipotent element\footnote{For the purpose of this definition, 
we shall suppose that $\ell$ is a ``good prime'' for $G^{\vee}$.}of $G^{\vee}(k)$.   Since $I_{\ell}$
has a unique cyclic quotient of order $p$, and all principal unipotent elements in $G^{\vee}(k)$
are conjugate, there exists only one conjugacy class of such maps $I_{\ell}
 \rightarrow G^{\vee}(k)$. 
\end{enumerate}

Let $\sigma: {\rm Gal} (\bar{\Q}/\Q) \rightarrow G^{\vee}(k)$ be a homomorphism. 
We say it is odd (resp. nearly ordinary of type $\bx$, resp. Steinberg) if its restriction
to $D_v$ for $v = \infty$ (resp. $v=p, v=\ell$) is. 
 
\begin{conj} \label{ashserre}    Let $S$ be a finite subset of places
of $\Q$, not containing $p$.   Suppose $\rho: \Gal \rightarrow  G^{\vee}(k)$
is an odd representation, nearly ordinary of type  $\bx$, Steinberg at
every place in $S$, and unramified outside $S$.  Put $N = \prod_{\ell \in S} \ell$. 

Then there is a Hecke eigenclass in $H^*(\Gamma_0(N), M_{\bx})$ that {\em matches $\rho$.}
\end{conj}

The meaning of ``matches'' is the usual compatibility\footnote{ To be precise:  
Let $s$ be a prime number distinct from $p$. 
Let $\mathscr{H}_s$ be the ``abstract'' algebra over $\Z$ generated by Hecke operators at $s$, i.e. the algebra of $\Z$-valued $\G(\Z_s)$-bi-invariant  compactly supported functions on $\G(\Q_s)$. 
Then the usual definition of the Satake isomorphism  (see \cite{GrossSatake} for discussion) actually defines an isomorphism $\mathscr{H}_s \otimes k \rightarrow \mathrm{Rep}(\check{G}) \otimes_{\Z} k$; it is important that $\G$ is simply connected for this remark, otherwise the half-sum of positive roots causes difficulty. } between Frobenius eigenvalues and Hecke eigenvalues. 
The idea of such a conjecture is due to A. Ash and others and was formulated precisely for $\mathbf{G} = \SL_n$ in \cite{Ash}. However, to our knowledge, it has not been formulated for a general group $\mathbf{G}$; {\em any error in this formulation is due solely to us.} We have formulated it only in the special case of Steinberg ramification at $S$
so that we could be precise about the level (in particular, to obtain a numerically falsifiable conjecture!) 

Suppose, for example, that $\mathbf{G} = \mathrm{SL}_n$; take $\mathbf{T}$ to be the
diagonal torus. 
For any integers $a_1 \geq a_2 \geq \dots$, 
 \begin{equation} \mathbf{x}_{a_1 a_2 \dots a_n}:    \left( \begin{array}{cccc} y_1 & 0 & \dots & 0 \\ 0 & y_2  & \dots &0  \\  0 & 0 & \ddots & 0 \\ 0 & 0 & \dots & y_n. \end{array}\right)  
 \mapsto \prod y_i^{a_i + n-i}\end{equation} 
 is a weight.  It is strongly regular if $a_1 > a_2 > \dots > a_n$. A representation $\rho: \Gal(\Qpbar/\Q_p) \rightarrow \PGL_n(k)$
is nearly ordinary of type $\mathbf{x}_{a_1, \dots, a_n}$ if its restriction to inertia is conjugate to
$$\left( \begin{array}{cccc} \omega^{a_1+n-1} & * & \dots & * \\ 0 & \omega^{a_2+n-2} & \dots &* \\  0 & 0 & \ddots & * \\ 0 & 0 & \dots & \omega^{a_n}.\end{array}\right),$$
where $\omega: I_p \rightarrow \mathbb{Z}_p^{\times}$ is the cyclotomic character. 

For $n=2$, the representation of $\mathbf{G} = \mathrm{SL}_2$ with highest weight $\bx-\rho$ is 
the symmetric $(a_1-a_2)$st power of the standard representation; 
thus, this corresponds to modular forms of weight 
$a_1-a_2+2$, and {\em strong regularity}
corresponds to weight $3$ or greater.

The following result is perhaps of independent interest; it will be of use in \S \ref{bharg}. 
\begin{prop} \label{MF}
Suppose $\bx$ strongly regular. For sufficiently large $p$,  the number of  nearly ordinary representations $\Gal(\overline{\Q_p}/\Q_p) \rightarrow G^{\vee}(k)$
of type $\bx$ equals $|G^{\vee}(k)|  \cdot  |N^{\vee}(k)|$. 
\end{prop}

The condition {\em sufficiently large $p$} arises in one way as follows: 
We wish to guarantee, for certain $m$ depending on $\bx$, that the $m$th power of the cyclotomic character $I_p \rightarrow \mathbb{F}_p^{*}$ is nontrivial. This will be so if $p-1 > m$, but not in general if $p$ is small.

Let us formulate an alternate phrasing, which will be convenient in the proof: 
Let $G_1, G_2$ be finite groups, and let $\Hom(G_1, G_2)^{\sharp}$
be the set of homomorphisms up to $G_2$-conjugacy. 
Then $\Hom(G_1, G_2)^{\sharp}$ is endowed with a natural measure:
the $G_2$-conjugacy class of a homomorphism $\alpha$ has mass $\frac{1}{|Z(\alpha)|}$;
here $Z(\alpha)$ is the centralizer of $\alpha$ in $G_2$. 
Alternately speaking, the measure (or {\em mass}) of $S \subset \Hom(G_1, G_2)^{\sharp}$
is the number of preimages in $\Hom(G_1, G_2)$, divided by $|G_2|$. 
With these conventions, 
we may rephrase the Proposition : the mass of  nearly ordinary representations of type $\bx$
equals $|N^{\vee}(k)|$.

\proof

Write, as shorthand, $D_p = \Gal(\Qpbar/\Q_p)$.  Fix an algebraic closure $\bar{k}$ of $k$.

Let $\psi_1, \psi_2$ be any two nearly ordinary representations of type $\bx$
valued in $B^{\vee}(k)$.  We claim that if $\psi_i$ are $G^{\vee}(k)$-conjugate,
then they are already $B^{\vee}(k)$-conjugate.  Suppose in fact
that $\Ad(g) \psi_1 = \psi_2$, for some $g \in G^{\vee}(k)$. 
The intersection $g B^{\vee} g^{-1} \cap B^{\vee}$ contains a maximal torus ${T^{\vee}}'$. 
 Let $N'$ be the normalizer of ${T^{\vee}}'$. All Borels containing a torus
 are conjugate under the normalizer of that torus, so there exists $n' \in N'(\bar{k})$ such that 
$g B^{\vee} g^{-1} = n' B^{\vee} n ' \- {}^{-1}$; since $B^{\vee}$ is its own normalizer,
it follows that there exists $b \in B^{\vee}(\bar{k})$ such that $g = n' b$.  

The inclusion into $B^{\vee}$ induces an isomorphism of
 ${T^{\vee}}'$ with $B^{\vee}/N^{\vee}$. 
Let $w$ be the automorphism of the latter group induced by $n'$.    
The composition
$$ I_p \hookrightarrow D_p \stackrel{\psi_j}{\rightarrow} B^{\vee}(k) \rightarrow  \left( B^{\vee}/N^{\vee}\right)(k)$$
is independent of $j$ (it is specified in terms of $\bx$). It is also $w$-fixed. 
  If $w$ is a nontrivial element of the Weyl group, 
this contradicts strong regularity for $q$ large.

 Thus, $n' \in B^{\vee}(\bar{k})$, which
shows that $g \in B^{\vee}(k)$.

We are reduced to proving that the total mass of nearly ordinary representations of type $\bx$
from $D_p$ to $B^{\vee}(k)$ equals $|N^{\vee}(k)|$; equivalently, 
 {\em the number of nearly ordinary representations $D_p \rightarrow B^{\vee}(k)$}
 equals $|N^{\vee}(k)| \cdot |B^{\vee}(k)|$.

  Let $\psi: D_p \rightarrow T^{\vee}(k)$ be any character extending $\psi_{\bx}: I_p \rightarrow T^{\vee}(k)$.  All such $\psi$ are unramified twists of one another; the total number of such $\psi$, then, 
  equals $|T^{\vee}(k)|$. 
  
On the other hand,  nearly ordinary representations $D_p \rightarrow B^{\vee}(k)$
  are precisely lifts (with respect to the isomorphism $B^{\vee}(k)/N^{\vee}(k) \cong T^{\vee}(k)$)  of some such $\psi$.
  We are therefore reduced to showing that
 {\em total number of lifts of any such $\psi: D_p \rightarrow T^{\vee}(k)$
equals $\frac{|N^{\vee}(k)| \cdot |B^{\vee}(k)|}{|T^{\vee}(k)|} = |N^{\vee}(k)|^2$. }

Now consider the lower central series for $N^{\vee} = U_0 \supset U_1 \supset U_2 \supset \dots$. 
Thus $\hatB/U_0(k)  = \hatT(k)$. 
Suppose we have fixed a lift $\psi_j$ of $\psi$ to $\hatB / U_j(k)$. 
We have an exact sequence of algebraic groups $$U_j/U_{j+1} \rightarrow \hatB/U_{j+1} \rightarrow \hatB/U_j. $$
This sequence is also exact on $k$-points. 
The group $X_j  = U_j(k)/U_{j+1}(k)$ is abelian; $D_p$ acts on it (via conjugation composed with $\psi_j$)
and the resulting action decomposes into character of the form $\alpha \circ \psi_j$, 
for various roots $\alpha : \hatT \rightarrow k^{*}$.

 By assumption of strong regularity, no $\alpha \circ \psi_j$
is either trivial or cyclotomic, at least for sufficiently large $p$.  Therefore, 
  $H^0(D_p, X_j) = H^2(D_p, X_j) = 0$.
The vanishing of $H^2$ shows that $\psi_j$ lifts to $\hatB/U_{j+1}$;
the vanishing of $H^0$ proves that $U_j/U_{j+1}$ acts freely by conjugation
on the set of lifts. 
The number of such lifts {\em up to $U_{j}/U_{j+1}$-conjugacy} is $H^1(D_p, X_j) = |X_j|$, by the local Euler characteristic formula \cite{Neukirch}; the total number of lifts is therefore 
$|X_j|^2$. 

The desired conclusion follows by iteration. \qed

\subsection{Bhargava's heuristics.} \label{bharg}

This section is rather speculative in nature. Our goal is to verify that Conjecture \ref{ashserre}
is compatible with the main result of this paper, Theorem \ref{main}. In other terms: does there exist a sufficiently large
supply of Galois representations to account for the (proven) exponential growth in torsion,
when $\delta = 1$?

Current technology is wholly inadequate to answer this question. We shall use heuristics
for counting number fields proposed by M. Bhargava. These suggest -- as we shall explain -- that
{\em the fraction of squarefree levels $N$ for which there is an eigenclass in $H^*(\Gamma(N), M_{\bx} \otimes \mathbb{F}_p)$ is $\gg p^{-\delta(G)}$.} 
Suppose now that $M_{\bx}$ is strongly acyclic; since the product $\prod_{p} p^{p^{-\delta(G)}}$
diverges when $\delta(G)=1$, this and Conjecture \ref{ashserre} imply the existence of levels $N$ for which
the torsion group has unbounded size. This is indeed compatible with the conclusion
of Theorem \ref{main}.

It is also possible to obtain support for Conjecture \ref{conjmain} via similar reasoning, but we do not discuss it here, focussing only on the case $\delta = 1$ where we can prove unconditional results about torsion.

 \subsubsection{}
 
Fix $\test$ a finite group.

 Let $\mathbf{L} = \prod_{v} \Hom(D_v, \test)^{\sharp}$ (the restricted 
 product, taken with respect to unramified representations).

 Let $\mathbf{G}$ be the $\test$-conjugacy classes of surjective homomorphisms
$\Gal(\Qbar/\Q) \rightarrow \test$. 

Thus there is a natural map $\mathbf{G} \rightarrow \mathbf{L}$. \footnote{The letters $G$ and $L$ stand for {\em global} and {\em local}.} 

The ``mass'' (see remarks after Proposition \ref{MF}) defines a measure on $\Hom(D_v, \test)^{\sharp}$.
For $v \neq \infty$, the set of unramified homomorphisms has mass $1$. Thus, 
this gives rise to a product measure $\mu$ on $\mathbf{L}$. 
Roughly speaking, Bhargava's heuristic suggests that {\em
the expected number of elements of $\mathbf{G}$ inside a large subset $S \subset \mathbf{L}$
should be approximately $\mu(S)$.}

Bhargava formulates this as a conjecture only in the following specific case: For $\test = S_n$, we 
can define the ``discriminant'' of any element of $\mathbf{G}$, i.e.,
the discriminant of the associated degree $n$ extension. This map $\mathbf{G} \rightarrow \mathbb{N}_{\geq 1}$ factors through $\mathbf{L}$. 
 Let $\mathbf{L}(N) \subset \mathbf{L}, \mathbf{G}(N) \subset \mathbf{G}$ be the subsets
 of discriminant $\leq N$. 
Then Bhargava conjectures that
\begin{equation} \label{bharg} |\mathbf{G}(N)| \sim  \mu \left(   \mathbf{L}(N)  \right) \ \ N \rightarrow \infty. \end{equation}
Bhargava proves this conjecture for $n \leq 5$.
 If we move away from $\test = S_n$, there are no shortage of simple examples (e.g., $\test = S_3^3$) where the heuristic fails in the precise form above. Our goal, however, is only to use it as an indication
 of the rough order of magnitude. 
 \subsubsection{}

Return to the setting of Conjecture \ref{ashserre}. Let $S$ be a finite subset of places of $\Q$, not containing $p$, and let $N(S)$ the corresponding squarefree integer - product of the primes in $S$.  
We are interested in the number $A(S)$ of Galois representations 
 $\rho:\Gal(\Qbar/\Q) \rightarrow G^{\vee}(k)$ as in Conjecture \ref{ashserre} -- 
 or, rather, the behavior of $A(S)$ on average with respect to $S$. 
  
For $v = \infty, p, \ell \neq p$ respectively, we define $M(v) \subset \Hom(D_v, G^{\vee}(k))^{\sharp}$ to consist of (respectively) odd, nearly ordinary of type $\bx$, and Steinberg-at-$\ell$ maps.  Let 
  $$\mu(S) =    \mass \ M(\infty) \cdot \mass \ M(p) \cdot \prod_{\ell \in S} \mass \ M(\ell).$$
  
  Now, Bhargava's heuristics suggest that:
  $$\sum_{N(S) < X} A(S) \sim \sum_{N(S) < X} \mu(S).$$
   
   We claim that  $\mu(S) = q^{-\delta(G)} (1+O(q^{-1/2}))$ for any such $S$; this follows from
the equality $2 \dim(N) + \rank(K) - \dim(G) = -\delta(G)$ and the following remarks: 
  \begin{enumerate}
  \item For $\ell \neq p$ a finite prime, $\mass \ M(\ell) = 1$: the Steinberg-at-$\ell$ maps form a single conjugacy class
  after restriction to $I_{\ell}$. 
  
  \item Let $\mathbf{O}$ be the algebraic subvariety of  $G^{\vee}$
consisting of odd involutions. We regard $\mathbf{O}$ as a scheme over $\mathbb{Z}$. It follows from Proposition \ref{odd} that
$\mathbf{O} \times_{\Z} \C $ is an irreducible variety of dimension
$\dim(N) + \rank(K)$.\footnote{Indeed, Proposition \ref{odd} shows that the complex points of $\mathbf{O} \times_{\Z} \C$ form a single $G^{\vee}(\C)$ orbit, and the dimension $D$ of the Lie algebra of 
the stabilizer of any point $\Theta$ satisfies: $2D- \dim (G) = {\rm trace} (\Theta )$. 
We get $D= \frac{\dim(G) + \rank(G) - 2 \rank(K)}{2}$ and the claim follows.}
  By the Lang-Weil estimates, 
 $|\mathbf{O}(k)| = q^{\dim(N) + \rank(K)} (1 + O(q^{-1/2})).$ 
 We conclude:
 $$\mathrm{mass}(M(\infty)) = q^{\dim(N) + \rank(K) - \dim(G)} (1+O(q^{-1/2})).$$

\item
For $v=p$, Proposition \ref{MF} gives
$\mathrm{mass}( M(p)) = q^{\dim(N ) }.$
\end{enumerate}

We conclude, then, that -- if we suppose valid a generalization of Bhargava's heuristics --  {\em the average value of $A(S)$, over $S$ with $N(S) < X$,
approaches $q^{-\delta(G)} (1+O(q^{-1/2}))$.}
 As we have already discussed at the start
of this section, this is ``compatible'' with the conclusion of Theorem \ref{main} in the case $\delta = 1$.  It also suggests that there should be a ``paucity'' of Steinberg torsion classes when $\delta \geq 2$.

{\em Remark.} One can also try to analyze the issue of torsion via eigenvarieties or Galois deformations. These 
(roughly speaking) parameterize lifts of a given mod $p$ torsion class or Galois representation, 
and conjectures concerning their dimension are therefore related to the 
 the likelihood that a mod $p$ class will lift modulo $p^2$. 
 In our (limited and possibly mistaken) understanding, conjectures such as those of \cite{CE, CalegariMazur}   suggest that the ``likelihood'' 
 that a   given  class modulo $p$ lifts modulo $p^2$ should be of order $p^{-\delta (G)}$. 
Also, Conjecture 6.1 of \cite{CE} also suggests that, when $\delta=1$, the torsion is largely
in ``the middle dimension,'' namely, in degree $\frac{\dim (S) -1}{2}$.

 \section{Combinatorial picture} \label{cp}

This section has a different flavour to the rest of the paper; it is more general, in that
we deal with an arbitrary cell complex; however, it is also less general in that we deal
with a {\em tower of coverings of a fixed space}, whereas many of our previous results
applied to any sequences of locally symmetric spaces with increasing injectivity radius. 

We have seen in \S \ref{circle2} that ``limit multiplicities formulas'' may be easier to prove in the
context of combinatorial torsion  than from the ``de Rham'' perspective. Such a phenomenon already 
appears in the classical case of normalized Betti numbers converging toward $L^2$-Betti numbers. The most 
general result, due to L\"uck \cite{Luck2}, is proved in the combinatorial setting and no general 
proof is known 
in the analytical one. 

\subsection{}

From now on let $K$ be a path-connected
finite cell complex having a universal cover $\widetilde{K}$ and fundamental group $\Gamma$, so that $K$ is the 
quotient of $\widetilde{K}$ by (left) deck transformations.  By ``the cover of $K$ corresponding to'' a subgroup $\Gamma' \leqslant \Gamma$, we mean the quotient $\widetilde{K}/\Gamma'$ together with the induced map to $K$. 

Now let $\Lambda$ be a quotient of $\Gamma$, i.e., we regard it as being equipped with a map $\Gamma \twoheadrightarrow \Lambda$.  Let $\widehat{K}$ be the covering space of $K$ corresponding to the 
kernel of $\Gamma \rightarrow \Lambda$;  let $\Lambda_N \subset \dots \subset \Lambda_1 = \Lambda$ be a decreasing sequence of finite index normal subgroups, with $\cap_{N=1}^{\infty} \Lambda_N = 1$,
and $K_N$ the cover corresponding to the preimage of $\Lambda_N$ in $\Gamma$. Denote by
$$T (K_N) = \prod \- {}^* | H^i (K_N)_{\tors}|,$$
the {\it torsion part of the cohomology}, and define the regulator $R^i(K_N)$ 
as in \eqref{RT}.

Let $C^j(\widehat{K})$ be the cellular cochain complex of $\widehat{K}$.
  Each $C^j$
is a finitely generated free $\Z[\Lambda]$-module.  
A cell complex for $K_N$ is given by $C^j(\widehat{K}) \otimes_{ \Z\Lambda} \Z[\Lambda/\Lambda_N]$. In all cases, we equip the complexes with the inner products in which the characteristic functions of distinct cells form an orthonormal basis.

Let $\Delta_j = d_{j-1} d_{j-1}^* + d_{j}^* d_j$; here $d_j^*$ is the adjoint of $d_j$ with respect to the inner product described. Each $\Delta_j$ is a $\Z[\Lambda]$-endomorphism of the 
finite free $\Z[\Lambda]$-module $C^j$. 

  In particular, if we fix a basis for $C^j$
as a left $\Z[\Lambda]$-module,   $\Delta_j$ is given by right multiplication by a unique
matrix with entries in $\Z[\Lambda]$.  We indeed fix such bases, and, in what follows,
regard $\Delta_j$ as an element of $\mathrm{Mat}_{n_j}(\Z[\Lambda])$.

\subsection{}
Let us discuss, first of all, the simple case where $\Lambda =\Z^m, \Lambda_N = N \Z^m$.

\subsubsection*{$\ell^2$-acyclicity.} 
For $(z_1, \dots, z_m) \in (\C^{\times})^m$, let $\Delta_j(z) \in \mathrm{Mat}_{n_j}(\C)$ be obtained from $\Delta_j$ via the homomorphism
$(n_1, \dots, n_m) \in \Lambda \mapsto \prod_{i} z_i^{n_i}$. We say that $\widehat{K}$ is $
\ell^2$-acyclic if $\det \Delta_j(z)$ -- which is
always a polynomial in $z_i^{\pm 1}$ -- is not 
identically vanishing for any $j$.  (This condition is equivalent to the vanishing of the reduced $\ell^2$-cohomology.)

Note that
$$\dim H^j (K_N,\C)= \sum_{\mathbf{z}^N = 1} \dim \ker \Delta_j(z).$$
The set of roots of unity contained in the zero locus of $\ker \Delta_j$ is contained
in a finite union of translates of proper subtori of $(S^1)^m$ (this is a simple example of a ``Manin-Mumford'' phenomenon; see \cite{ClairWhyte} for a discussion of this in topological context) so it follows that, if $\widehat{K}$ is $\ell^2$-acyclic,  then
  \begin{equation} \label{manman} \dim H^j(K_N,\C) \leq A N^{m-1},\end{equation} the constant $A$ depending only on $K$.

\subsubsection*{$\ell^2$-torsion.} 
We define it as:
\begin{equation}\label{taudef} \tau^{(2)}(\widehat{K};\Lambda) := \sum_{j} (-1)^{j+1} j  \int_{\mathbf{z} \in (S^1)^m} \log |\det \Delta_j (\mathbf{z}) | d\mathbf{z},\end{equation} 
the integral being taken with respect to the invariant probability measure on the compact torus $(S^1)^m$. 

For example, let $k$ be an oriented knot with exterior $V = {\Bbb S}^3 - {\rm int} \ N(k)$, where $N(k)$ is a regular 
neighborhood of $k$. The meridianal generator of the knot group $\pi_1 ({\Bbb S}^3 - k)$ represent 
a distinguished generator $t$ for its abelianization. We identify this generator with the standard generator
$1$ of ${\Bbb Z}$.  Fix $K$ a triangulation of $V$ as a finite polyhedron and let $\widehat{K}$ be its maximal Abelian cover corresponding to the kernel of $\pi_1 ({\Bbb S} - k) \rightarrow {\Bbb Z}$.
Let $\Delta$ be the Alexander polynomial of $k$. It is never identically vanishing and $\widehat{K}$
is $\ell^2$-acyclic, see \cite{Milnor} for more details. It is well known see e.g. \cite[(8.2)]{LiZhang} that
\begin{eqnarray} \label{L2torsionk}
\tau^{(2)} (\widehat{K} ; {\Bbb Z} )  = -  \log M(\Delta ). 
\end{eqnarray}

\begin{thm} \label{Papprox} (Growth of torsion and regulators in abelian covers). 
Notation as above; suppose that $\widehat{K}$ is $\ell^2$-acyclic. Then
\begin{equation} \label{rconv} \frac{\log R^i(K_N)}{[\Lambda:\Lambda_N]} \rightarrow 0, \ \ 0 \leq i \leq \dim(K). 
\end{equation}
Moreover, if $m=1$, we have
\begin{equation} \label{conv} \frac{\log T(K_N)}{[\Lambda:\Lambda_N]} 
\rightarrow -\tau^{(2)}, \end{equation} 
where the $\tau^{(2)}$ is the $\ell^2$-torsion. 
\end{thm}

 Applied to cyclic covers of a knot complement,  Theorem \ref{Papprox} translates into the theorem of Silver and Williams \cite{SilverWilliams} mentioned in the introduction.
Likely one could prove a version of \eqref{conv} for $m > 1$ replacing $\lim$ by $\limsup$;
also, one can establish \eqref{conv} for general $m$ if we suppose that each $\det \Delta_j$ 
is everywhere nonvanishing on $(S^1)^m$, this being the analog of ``strongly acyclic.'' 

\proof 
\eqref{rconv} follows from the subsequent Proposition \ref{subsec:CR}, taking into account \eqref{manman} and the fact that $\Lambda/\Lambda_N$ is abelian.

Now apply \eqref{RT} and \eqref{dLap}. They imply, together, that:
$$\log  \left( \frac{\prod_{i}^* R^i(K_N)}{ T(K_N) }  \right) =  \frac{1}{2}\sum_{j} (-1)^{j+1} j \left( \sum_{\zeta^N = 1}  \log \det{} ' \Delta_j(\zeta) \right). $$
\eqref{conv} then follows from \eqref{mahler}:
 
Write $\Delta_j(z) = a \prod_{j} (z-z_j)$. 
Note that $z \mapsto \det \Delta_j(z)$ is non-negative on the unit circle and has only finitely many zeroes amongst roots of unity, 
so 
$$\sum_{\zeta^N = 1}  \log \det{}' \Delta_j(\zeta)  -  N \log a - \sum_{z_j^N \neq 1} \log |z_j^N-1| $$
is bounded as $N \rightarrow \infty$. 
Thereby, $\frac{1}{N} \sum_{\zeta^N = 1}  \log \det{}' \Delta_j(\zeta) $ approaches
$\log a + \sum_j \log^+ |z_j|$, which is the integral of $\log \det \Delta_j(z)$ over $S^1$. 
 \qed

\begin{prop} \label{subsec:CR}
Let $\delta_N$ be the sums of squares of all the dimensions of $\Lambda/\Lambda_N$-representations occuring in $H^i(K_N, \C)$ (counted without multiplicity). 
Then: 
$$
   \left| \frac{\log R^i}{\delta_N}\right|
 \leq A \log [\Lambda:\Lambda_N] + B,$$
 the constants $A,B$ depending only on $K$.   
\end{prop}
Note that the Proposition does not require $\Lambda/\Lambda_N$ to be abelian. 
  The basic idea is to use the action of $\Lambda/\Lambda_N$ to split the cochain complex of $K_N$ into two pieces: an acyclic piece, and a piece contributing all the cohomology. 
  Under favorable circumstances, the latter piece is small, and it follows that $R^i$ is small.  
 In practice, we cannot literally split the complex into a direct sum of $\Z$-subcomplexes; nonetheless, we can do so to within a controlled error.

More specifically, the proof will follow from the string of subsequent Lemmas:  the bound $\frac{\log R^i}{\delta_N} \geq \cdots$ follows
by specializing Lemma \ref{gaction} to the case of the cochain complex of $K_N$;
the inequality $-\frac{\log R^i}{\delta_N} \geq \cdots$
follows from applying that Lemma to the dual complex and using Lemma \ref{duality}.  

We say that a metrized lattice is {\em integral} if the inner product $\langle - ,  - \rangle$ takes integral values on $L$. 
Note that a sublattice of an integral lattice is also integral. 
  
  \begin{lem*}
  Any integral lattice has volume $\geq 1$. 
  \end{lem*}

\begin{lem*} Let $f:A_1 \rightarrow A_2$ be a map of integral lattices of dimensions $\leq n$. 
Suppose every singular value of $f$ is $\leq M$, where we suppose $M \geq 1$.   Then:
 $$1  \leq \vol(\ker f)  \leq  \vol(A_1) M^n,$$
 and the same inequality holds for $\vol(\image f)$. 
\end{lem*}
\proof This follows from the prior Lemma and \eqref{erasmus}. \qed

\begin{lem*}
Let $\Abull$ be a complex, as in \eqref{abullet}. Suppose every $A^i$
is an integral metrized lattice (not necessary of volume one). Let $\nu > 1 , M > 1, D$ be such that:
\begin{enumerate}
\item Each $A^i \otimes \Q$ has a basis consisting of elements of $A^i$ of length $\leq \nu$; 
\item $\dim(A^i) \leq D$ for all $i$;
\item All differentials $d_i: A^i \rightarrow A^{i+1}$ have all singular values $\leq M$.
\end{enumerate}
Then: 
 \begin{equation} \label{regbound1} R^i(\Abull) \geq  (M   \nu )^{-D}   \end{equation} 
 \end{lem*}
 
 \proof 

Hadamard's inequality implies that $\vol(A^i) \leq \nu^D$ (indeed, 
if $x_1, \dots, x_r \in A^i$ form a $\Q$-basis, then
$\vol(A^i) \leq \|x_1 \wedge \dots \wedge x_r\|$, whence the result). 
The result follows from the prior Lemma and \eqref{volquot}.

\qed 
\begin{lem}\label{gaction} Notations as in the prior lemma; let $G$ be a finite group of order $|G|$ acting\footnote{We understand this to mean that the action on $\Abull \otimes \R$ is isometric.}  on $\Abull$.  
  Let 
$\Xi$  be the set of all characters of all irreducible $G$-representation that occur in $H^i(\Abull \otimes \C)$ and let $D$ be the dimension of the $\Xi$-isotypical subspace
of $A^i$. Then 
$$R^i(\Abull)  \geq (M \nu  |G|^5 )^{-D}.$$
\end{lem}

\proof

Set 
$$e_{\Xi} = \sum_{\chi \in \Xi} \chi(1) \sum_{g \in G} \overline{ \chi(g)} g  \in \C[G].$$
Then, in fact, $e_{\Xi} \in \Z[G]$: its coefficients are algebraic integers that are fixed under 
Galois conjugacy. 
Moreover,  $e_{\Xi}^2 =  |G| e_{\Xi}$; indeed, $e_{\Xi}/|G|$ realizes a projection onto the $\Xi$-isotypic space of any $G$-representation.

Since $|\chi(g)| \leq |G|$ for all $g$,  $|\Xi| \leq |G|$, and each $g \in G$ acts isometrically on $\Abull$, we see that $e_{\Xi}$ acting on $\Abull$
increases norms by at most $|G|^4$.

Let $\pi: \Abull \rightarrow \Abull$ be the endomorphism induced by $e_{\Xi}$;
let $\Bbull$ be the image of $\pi$, and 
 $\iota: \Bbull \hookrightarrow \Abull$ the inclusion.
 Any ``harmonic form'' in $\Abull \otimes \R$ is fixed by the projection $e_{\Xi}/|G|$
 and so belongs to $\Bbull \otimes \R$. 
Therefore  $\iota$ induces an isometric isomorphism
$$\iota_{*\R}: H^*(\Bbull, \R) \rightarrow H^*(\Abull, \R).$$

We are going to show that $\iota_{*\R} \pi_{*\R}$ -- an endomorphism of $H^*(\Abull,\R)$ --  is multiplication by $|G|$. 
Indeed, let $x \in A^j $ be a cycle, so that $dx = 0$. There exists a cycle $y \in B^j \otimes \R$ 
such that $\iota y = x + dz$, for some $z \in A^{j-1}\otimes \R$. 
Write $[x],[y]$ for the corresponding cohomology classes with real coefficients, so that $\iota_{*\R} [y] = [x]$. 
Now $\iota_{*\R} \pi_{*\R}[x]$ equals
 $\iota_{*\R} \pi_{*\R} \iota_{*\R} [y] = |G| \iota_{*\R} [y]  = |G| \cdot  [x]$.  The claim follows.

This implies that the image of $|G| \cdot H^i(\Abull, \Z)$ inside $H^i(\Abull, \R)$
is contained in the image of $H^i(\Bbull, \Z )$ inside $H^i(\Bbull, \R) \stackrel{\sim}{\rightarrow} H^i(\Abull, \R)$. Thus, 
\begin{equation}
 \label{newreg} R^i(\Abull) \geq |G|^{-\dim H^i} R^i(\Bbull). \end{equation} 
By definition of $B^i$ as the image of the endomorphism $\pi$, 
which increases norms by at most $|G|^4$, $B^i \otimes \Q$ is generated as a $\Q$-vector space by elements of $B^i$ of length 
$\leq |G|^4 \nu$.   Finally, $\dim B^i \leq D$.
Combining \eqref{newreg} and \eqref{regbound1} gives the desired conclusion. 
\qed

\begin{lem} \label{duality}  
Let $\Abull$ be a complex, as in \eqref{abullet}, and let $\widehat{\Abull}$ be the
dual complex
{\em dual complex} $$0 \leftarrow \widehat{A}^0 \leftarrow \widehat{A}^1 \leftarrow \dots$$
where $\widehat{A^j} = \Hom(A^j, \Z)$, endowed with the dual metric. 
Then the regulator $\widehat{R}^j$ of $\widehat{\Abull}$ satisfies:
$$\widehat{R}^j \cdot R^j = 1.$$
\end{lem}
\proof
In fact, the pairing between $A^j$ and $\widehat{A}^j$ induces a perfect
pairing between the cohomology groups of $\Abull \otimes \C$ and $\widehat{\Abull} \otimes \C$,  and the image of the cohomology of $\Abull$ and $\widehat{\Abull}$
projects to lattices in perfect pairing with respect to this duality. 
\qed 

We have now concluded the proof of Theorem \ref{Papprox}; let us describe an application:

\begin{cor}
Let $V$ be a compact $3$-manifold which fibers over the circle. Let $S$ be a fiber, $f:S \rightarrow S$
the monodromy map and $P_f$ the characteristic polynomial of the linear map induced by $f$ on $H_1 (S)$. The fibration over the circle induces a map $\pi_1 (V) \twoheadrightarrow \Z$. We denote by $V_N$ the corresponding $N$-fold covering of $V$. Then:
$$\lim_{N \rightarrow +\infty} \frac{\log |H_1 (V_N)_{\rm tors} |}{N} = \log M (P_f ).$$
\end{cor}
\proof It is a theorem of L\"uck \cite[Theorem 1.40]{Luck} that the infinite cyclic cover $\widehat{V}$ associated
to $\pi_1 (V) \twoheadrightarrow \Z$ is $\ell^2$-acyclic. We may thus apply Theorem \ref{Papprox}. Here - as in the case of knots complements - the $\ell^2$-torsion of $\widehat{V}$ is equal to $-\log M (\Delta)$
where $\Delta$ - the natural generalisation of the Alexander polynomial - is equal to $P_f$. This proves
the corollary.
\qed

\subsection{Asymptotics of torsion.}  \label{as-co-to}
We now work with general $\Lambda$.  It is possible to define
the combinatorial $\ell^2$-torsion $\tau^{(2)}(\widehat{K};\Lambda)$ of $\widehat{K}$ in a manner generalizing
\eqref{taudef} in the case $\Lambda =\Z$. 

\medskip

\noindent
{\em Question.} Supposing $\widehat{K}$ has trivial $\ell^2$-homology; under what circumstances does the sequence
$$\frac{\log T(K_N)}{[\Gamma : \Gamma_N]} $$
 converge to $-\tau^{(2)} (\widehat{K} ; \Lambda )$ as $N \rightarrow +\infty$ ?

 \medskip 
 
 This is true if $\widehat{K}$ is ``strongly acyclic,'' in that the eigenvalues of each $\Delta_j$ on each $K_N$
 are uniformly separated from zero. 
 We do not know any example where this is so besides those already discussed.

\section{Examples} 
\label{examples}

\subsection{Existence of strongly acyclic bundles.} \label{SAexist}

Let notation be as in Theorem \ref{main}.  We shall prove that
strongly acyclic bundles {\em always exist when $\delta = 1$.}

Let $\mathbf{T} \subset \mathbf{G}$ be a maximal torus. Let $E$ be a Galois extension of $\Q$ splitting $\mathbf{T}$.  Let $\mathbf{T}_E = \mathbf{T} \times_{\Q} E$, and let $X^*$ resp. $X^*_+$
be the character lattice of $\mathbf{T}_E$, resp. the dominant characters. 
For each $x \in X^*_+$ there is associated a unique irreducible representation of $\mathbf{G} \times_{\Q} E$, denoted $\rho_x$, with highest weight $x$.  There is an associated homomorphism
$$  \mathbf{G} \rightarrow \mathrm{Res}_{E/\Q} \GL(W_x) \hookrightarrow \GL(\mathrm{Res}_{E/\Q} W_x),$$
where, by $\mathrm{Res}_{E/\Q} W_x$, we mean $W_x$ considered as a $\Q$-vector space. 

Let $\tilde{\rho}_x$ be this representation of $\mathbf{G}$ on $\tilde{W}_x : = \mathrm{Res}_{E/\mathbb{Q}} W_x$. 
Then $\tilde{W}_x \otimes_{\Q} \C$ decomposes into irreducible $\mathbf{G}(\C)$-representations; their highest weights are obtained from $x$ by applying the various embeddings $E \hookrightarrow \C$.

We claim that, so long as $x \in X^*_+$ avoids a finite union of proper hyperplanes, any stable $\Z$-lattice
in $\tilde{W}_x$ will be strongly acyclic.

  In fact, 
the complexified Cartan involution induces a certain $W$-coset 
 $\mathcal{W} \subset \mathrm{Aut}(X^*(\T_{\C}))$: Pick any Cartan-stable maximal  torus $\mathbf{T}'$ of $\mathbf{G}$ defined over $\R$;
the complexified Cartan involution induces an automorphism of $X^*(\mathbf{T}'_{\C})$, 
the latter is identified with $X^*(\mathbf{T}_{\C})$, uniquely up to the action of a Weyl element.
According to the Lemma below, each $\omega \in \mathcal{W}$ acts nontrivially.
Now, by \eqref{SA}, we may take the ``finite union of hyperplanes'' to be the preimages
of the $\{\mathrm{fix}(\omega): \omega \in \mathcal{W} \}$ under the various maps $X^* \rightarrow X^*(\mathbf{T}_{\C})$
induced by the various embeddings $E \hookrightarrow \C$.

    \begin{lem*} Suppose that $\delta = 1$.  Let $\theta: \mathfrak{g} \rightarrow \mathfrak{g}$ be
    (the complexification of) a Cartan involution. Let $\mathfrak{t}$ be any $\theta$-stable 
    Cartan subalgebra. Then the action of $\theta$ on $\mathfrak{t}$ does not coincide with any Weyl element. In particular, $\theta$ acts nontrivially on the set of isomorphism
    classes of irreducible $\mathfrak{g}$-modules. 
\end{lem*} 
 
\proof 
It suffices to check the case when $\mathfrak{g}_{\R}$ is isomorphic to one of the simple real Lie algebras $\mathfrak{sl}_3$ and $\mathfrak{so}(p,q)$  for $pq$ odd.   Now suppose the assertion is false;
we may multiply $\theta$ by an inner automorphism to obtain $\theta'$ which acts
trivially on $\mathfrak{t}$ {\em and} on every positive simple root space. 
Using the Killing form, we deduce that the action of $\theta'$ on every negative simple root space is trivial. Since these, together, generate $\mathfrak{g}$,  we conclude that $\theta'$ is trivial,
and $\theta$ was inner: a contradiction, for in both cases the Cartan involution is not inner. 
\qed

For example, if $\mathbf{G}$ is given by the units in a nine-dimensional division algebra  $D$
over $\Q$, then a congruence lattice in $\mathbf{G}$ is given by $\mathfrak{o}_D^{\times}$, the units in a maximal order of $D$; and a strongly acyclic $\mathfrak{o}_D^{\times}$-module
is given by $\mathfrak{o}_D$ itself.  Indeed, in this case -- taking for $E$ a cubic extension of $\Q$ that splits $D$ -- the representation ``$\tilde{W}_X \otimes \C$,'' notation as previous,
is the the sum of three standard representations of $\G(\C) \cong \SL_3(\C)$; moreover,
the isomorphism class of the standard representation is not fixed by the Cartan involution; it is interchanged with its dual.

\subsection{Hyperbolic $3$-manifolds  and the adjoint bundle.} \label{adjoint}
There is a particularly interesting strongly acyclic bundle that exists for certain arithmetic hyperbolic $3$-manifolds:

Let $B$ be a  quaternion division algebra over an imaginary quadratic field, and $\mathfrak{o}_B$
a maximal order.  Then $\mathfrak{o}_B^{\times}$ embeds into $\PGL_2(\C)$, and acts on $\mathbb{H}^3$; let $M$ be the quotient, and $\tilde{M}$ any covering of $M$. 

Let $L$ be the set of trace-zero elements in $\mathfrak{o}_B$, considered as a $\pi_1(M)$-module via conjugation. Then $H^1(\tilde{M}, L \otimes \C) = 0$. This is a consequence of Weil local rigidity. 
In fact, the module $L$ is strongly acyclic, and explicating the proof of the Theorem shows that:
\begin{equation} \label{thm-rough}   \frac{\log | H_1(\tilde{M}, L)|}{\vol(\tilde{M})}
\longrightarrow \frac{1}{6\pi}.    \end{equation} 

This has the following consequence: 
Although the defining representation $\pi_1(\tilde{M}) \rightarrow \SL_2$
does not deform over the complex numbers, it does deform {\em modulo $p$} for many $p$:
indeed, if $p$ divides the order of $H_1(\pi_1(M), L)$, it means precisely that
there is a nontrivial map
$$\pi_1(M) \rightarrow \SL_2(\mathbb{F}_p[t]/t^2).$$ 

 It would be interesting if the existence of ``many'' such quotients
shed any light on the conjectural failure of property $\tau$; cf. \cite{LubotzkyWeiss}.

 {\em Remark.} For {\em any} compact
  hyperbolic manifold $M$, we may consider the three-dimensional flat complex bundle
  defined by composing
$\rho: \pi_1(M) \rightarrow \SL_2(\C)$ defining the hyperbolic structure,  with the adjoint action
of $\SL_2(\C)$. Again, the corresponding local system is acyclic. However, in general, this representation {\em does not have an integral structure}.  
It can sometimes be defined over the ring of integers of a number field larger than $\Q$;
but the resulting local system of $\Z$-modules is not {\em strongly} acyclic.

\subsection{Lifting torsion:  hyperbolic $n$-manifolds, $n > 3$.} \label{pullbacks}
Combined with geometric techniques, one may use Theorem \ref{main}
to obtain torsion even in certain (nonarithmetic, nonexhaustive) sequences
with $\delta = 0$. 

Let $F$ be a totally real number field and $(V,q)$ be an anisotropic quadratic space over $F$, of signature 
$(n,1)$ over one infinite place and definite over all other infinite places. Let $U \subset V$ be a $F$-rational
non-degenerate subspace of signature $(3,1)$ over one place at infinity. Let $\sigma$ denote the involutive linear map from $V$ to itself which acts as the identity on $U$ and as minus the identity on 
$U^{\perp}$.

It is proved in \cite{BHW} that there exists a -- possibly non-congruence -- lattice $\Gamma \subset 
\SO (q,F)$ such that $\Gamma$ retracts onto $\Gamma^{\sigma}$. We denote by 
$$r : \Gamma \rightarrow \Gamma^{\sigma}$$
the retraction.

The group $\Gamma^{\sigma}$ is a lattice in $\SO (U) \cong \SO_{3,1}$. We let $(L, \rho)$ be a strongly acyclic
$\Z [\Gamma^{\sigma} ]$-module. Composing $\rho$ with the retraction $r$ turns $L$ into a 
$\Z [\Gamma ]$-module that we still denote $\tilde{L}$. It is not arithmetic. 

\begin{prop*} 
1) There exists a decreasing sequence of finite index subgroups $\Gamma_N \leq \Gamma$ so that $|H_1 (\Gamma_N , \tilde{L})_{\rm tors} |$ grows exponentially with $[\Gamma : \Gamma_N]$.

2) There exists a decreasing sequence of finite index subgroups $\Gamma_N \leq \Gamma$  
such that $\cap_N \Gamma_N = \{ 1 \}$ and
$$ \liminf \frac{\log |H_1 (\Gamma_N , \tilde{L})_{\rm tors} |  } {    {\rm vol} (\Gamma_N^{\sigma} \backslash \H^3 )} > 0.$$
\end{prop*}
\proof
1) Theorem \ref{main} implies that $\Gamma^{\sigma}$ contains a decreasing sequence of finite index 
subgroups $\Gamma_N^{\sigma} \leq \Gamma^{\sigma}$ such that $|H_1 (\Gamma_N^{\sigma} , L)_{\rm
tors}|$ grows exponentially with the index $[\Gamma^{\sigma} : \Gamma_N^{\sigma}]$. 
Let $\Gamma_N$ be the preimage of $\Gamma_N^{\sigma}$ by the retraction map $r$. Then 
$$H_1 (\Gamma_N , \tilde{L} ) \twoheadrightarrow H_1 (\Gamma_N^{\sigma} , L) $$
and the conclusion follows.

2) Let $\Gamma_N \leq \Gamma$ be any decreasing sequence of finite index subgroups such
that $\cap_N \Gamma_N = \{ 1 \}$. It follows from \cite{BHW} that the map $r$ induces a
virtual retraction from $\Gamma_N$ onto $\Gamma_N^{\sigma} = \Gamma_N \cap \Gamma^{\sigma}$. Replacing each $\Gamma_N$
by a finite index subgroup we may thus assume that the retraction map $r$ induces a retraction
from $\Gamma_N$ onto $\Gamma_N^{\sigma}$. This yields injections:
$$H_1 (\Gamma^{\sigma}_N , L) \hookrightarrow H_1 (\Gamma_N , \tilde{L}).$$
Now $\cap_N \Gamma_N^{\sigma} = \{ 1 \}$ and $L$ is strongly acyclic. So that Theorem \ref{main} implies that 
$$\liminf \frac{ \log |H_1 (\Gamma_N^{\sigma} , L )_{\rm tors} | }{{\rm vol} (\Gamma_N^{\sigma} \backslash \H^3 )} >  0. $$
\qed

Note that in case 1), the intersection $\cap_N \Gamma_N$ is non trivial. It is equal to the kernel of $r$.
Case 2) is reminiscent from the results of Borel, Labesse and Schwermer \cite{BLS} who construct 
non-trivial cohomology classes for some series of examples of arithmetic groups as lifts from discrete 
series in a smaller group.

 \subsection{The hyperbolic plane; Shimura varieties.} 

In case $G= \SL_2 (\R)$, any torsion-free lattice $\Gamma \leqslant G$ has 
torsion-free cohomology; indeed, the homology of $\Gamma$ is identified
with the homology of the Riemann surface $\Gamma \backslash \mathbb{H}^2$.

More generally, it is \underline{believed} that the cohomology of Shimura varieties
(all of which have $\delta =0$, for example, $\G = \mathrm{Sp}_{n}(\R)$, or $\G = \mathrm{O}(2n, m)$, or $\G = U(n,m)$)
should have very little torsion -- precise statements in a related vein are proposed
in \cite{CalegariEmerton}. 

Evidence in this direction is provided by \cite{Boyer, Dimitrov}.

\subsection{Good groups} Following Serre \cite{Serre} we say that a group $G$ is {\it good}~\footnote{This terminology
was introduced by Serre in the course of an exercice. Unfortunately it has now become standard.} if the homomorphism of cohomology groups 
$H^n (\widehat{G} , M ) \rightarrow H^n (G,M)$
induced by the natural homomorphism $G \rightarrow \widehat{G}$ of $G$ to its profinite completion
$\widehat{G}$ is
an isomorphism for all $n$ and every finite $G$-module $M$.

Recall two simple facts, see \cite[Exercices 1 and 2 (b), p. 13]{Serre}. 
\begin{center}
{\it
\begin{enumerate}
\item Goodness is preserved by commensurability.
\item If $E/N$ is an extension of a good group $G$ by a good group $N$ such that the cohomology 
groups $H^q (N,M)$ are finite for all $q$ ($q \in {\Bbb N}$) and all finite $E$-modules  $M$, then $E$
is good.
\end{enumerate}}
\end{center}

If a lattice $\Gamma$ is good then the projective limit
$$\lim_{\substack{\longrightarrow \\ \Gamma '}} H^j ( \Gamma'  , \Z / n \Z )$$
over all its finite index subgroups $\Gamma'$ is always zero; in other terms,
given any class $\alpha \in H^j(\Gamma, M)$ there exists a finite index subgroup
$\Gamma_0 \leqslant \Gamma$ so that $\alpha | \Gamma_0$ is trivial.

Arithmetic lattices which satisfy the congruence subgroup property (CSP) are not good. On the
other hand Thurston has conjectured that lattices in $\SL_2 (\C)$ are commensurable with
the fundamental group of a $3$-manifold fibering over the circle. Such a group is an extension of
$\Z$ by either a free group or a surface group. Since these groups are good, it would follow from
Thurston's conjecture (and the above two simple facts) that lattices in $\SL_2 (\C)$ are good.

In \cite{Agol} Agol introduces a new residual condition (RFRS) under which he proves Thurston's conjecture. Commonly occurring groups which are at least virtually RFRS include surface groups
and subgroups of (abstract) right angled reflection groups. Since it follows from \cite{BHW} that 
arithmetic hyperbolic groups defined by a quadratic form are virtual subgroups of (abstract) right angled reflection groups, on gets~\footnote{A particular case is proved in \cite{GJZZ}.}:

\begin{center}
{\it Arithmetic lattices of $\SL_2 (\C)$ which are defined by a quadratic form are good.}
\end{center}

\subsection{Slow growth of coinvariants} \label{H0bound} 
In this section we show (in a special case)   -- in the notation of Theorem \ref{main} -- that  the growth
of the orders of $H_0(\Gamma_N, M)_{\tors}$ and $H_{\dim(S)-1}(\Gamma_N, M)_{\tors}$
is at most polynomial in the index.  This is compatible with Conjecture \ref{main}, 
although it is too easy a case to be serious evidence in that direction;
we include it largely for completeness. 
 
 We carry this out for a specific family of congruence subgroups;
it is very likely a similar result holds in greater generality.
 Suppose that $\G, M$ are as in that theorem, $\Gamma$ a congruence subgroup of $\G(\Q)$, and $\rho$ the representation of $\G$
 on $M \otimes \Q$.  We suppose that $\rho$ is {\em faithful}, that it has no trivial factors, 
 and that $\Gamma$ is Zariski-dense in $\G$.  These conditions guarantee that
 $H_0(\Gamma', M)$ is torsion for any finite index subgroup $\Gamma' \leqslant \Gamma$. 
 
  Consider
 the family of congruence subgroups $$\Gamma(N) := \Gamma \cap \rho^{-1}(1 + N  \ \End (M )),$$
 where we regard $1 + N  \ \End(M)$ as a (principal congruence) subgroup of $\GL(M)$. 
 Under these circumstances we show that:
 \begin{quote}The orders of the (finite) groups
 $H_0(\Gamma(N), M)$ and $H_{\dim(S)-1} (\Gamma(N), M)$ are bounded
 by a polynomial in $N$. 
 \end{quote}
 
 We shall prove the result concerning $H_0$, the result about $H_{\dim(S)-1}$ following by a dual argument.  Note that shrinking $\Gamma$ can only increase the size of $H_0$; 
we begin by replacing $\Gamma$ by a finite index subgroup so that its closure in $\G(\adele_f)$
 is a product group, i.e., of the form   $\prod_p K_p$, with $K_p \subset \G(\Q_p)$
 a compact subgroup.

Considering the exact sequence
$H_0(\Gamma (N),  M) \stackrel{\times \ell^r}{\longrightarrow} H_0(\Gamma (N), M) \rightarrow H_0(\Gamma (N),  M/\ell^r M)$
we see that 
\begin{equation} \label{product} 
| H_0(\Gamma (N) , M) | =  \prod_{\ell} | H_0(K_{\ell} [N],  M \otimes \Z_{\ell})|. 
\end{equation}
Here, $K_{\ell}[N] =  \rho(K_{\ell})  \cap 1 + N \End(M \otimes \Z_{\ell})$.
In particular, $K_{\ell}[N] = K_{\ell}[1]$ if $\ell $ does not divide $N$. Consequently, 
\begin{equation} \label{productI} 
| H_0(\Gamma (N) , M) | \leq   |H_0(\Gamma(1), M)| \cdot  \prod_{\ell | N} | H_0(K_{\ell} [N],  M \otimes \Z_{\ell})|. 
\end{equation}
Fix now a prime $\ell$ dividing $N$. 
Let $s$ be the $\ell$-valuation of $N$ if $\ell$ is odd;
if $\ell =2$, let $s = \max(3, \mbox{$2$-valuation of $N$})$. 
 Note that $K_{\ell} [N] \supseteq  K_{\ell}[\ell^s ]$.

Let $\mathfrak{g}$ be the Lie algebra of $\rho(\mathbf{G})$; 
we identify it with a subspace of $\mathrm{End}(M_{\Q})$, and, 
because $\rho$ has no invariant subspaces, 
  $\mathfrak{g}  M_{\Q} =  M_{\Q}$.    (Here we write $\mathfrak{g} M_{\Q}$
  for the image of the action map $\mathfrak{g} \times M_{\Q} \rightarrow M_{\Q}$;
  we use similar notation in what follows.)
Write $\mathfrak{g}_{\mathbb{Z}} = \mathfrak{g} \cap \End(M)$. Then 
the index of $[M: \mathfrak{g}_{\Z} M] $ is finite. 

Set  $\mathfrak{g}[\ell^s] = \mathfrak{g} \cap \ell^s \mathrm{End}(M)$ and $M_{\ell} = M \otimes \Z_{\ell}$.
Let $U_{\ell}$ be  the space 
 spanned by $\{ g v - v \; : \;  g \in K_{\ell}[\ell^s], v \in M_{\ell}\}$. We claim
 $$ U_{\ell} \supset \{X v \; : \; X \in \mathfrak{g} [\ell^s], v \in M_{\ell}\}.$$ 
 
Indeed, given $X \in \mathfrak{g} [\ell^s]$,  set $g = \exp(X) = \sum_{i \geq 0} \frac{X^i}{i!} \in GL(M_{\ell})$; 
the series converges in $\End(M_{\ell})$. One checks that in fact $g \in K_{\ell}[\ell^s]$.
Then, for $v \in M_{\ell}$, 
$$Xv = \sum_{i \geq 1} (-1)^i \frac{(g-1)^i v }{i}   = (g-1) \left( \sum_{i \geq 1}  (-1)^i\frac{(g-1)^{i-1}v}{i}  \right).$$

 Since $\frac{(g-1)^{i-1} v}{i} \in M_{\ell}$ for each $i \geq 1$, we see 
indeed that $X v \in U_{\ell}$.  Since $H_0(K_{\ell}[N], M) $ is isomorphic
to a quotient of $M/U_{\ell}$, we deduce
 $$ | H_0(\Gamma(N), M)| \leq |H_0(\Gamma(1), M)| \cdot   (8N)^{\mathrm{rank}(M)^2} \cdot  [M: \mathfrak{g}_{\Z} M].$$

\section{Questions}  \label{s:quest}

\subsection{Trivial coefficients.} 
For simplicity, we confine ourselves to hyperbolic $3$-manifolds and phrase our results in terms of homology, since in this case the interesting torsion is in $H_1$, which is simply the abelianized fundamental group. 

If one considers trivial coefficients, the Cheeger-M{\"u}ller theorem expresses the {\em product}:
$$| H_1(M, \Z)_{\tors}| . R $$
as a regularized Laplacian determinant; here $R$ is defined as follows: take $\omega_1, \dots, \omega_N$ an $L^2$-basis of harmonic $1$-forms, 
$\gamma_1, \dots, \gamma_N$ a basis for $H_1(M, \Z)$, and put
$R  := \det  \left(  \int_{\gamma_i} \omega_j  \right)_{1 \leq i,j \leq N}^{-2}.$

Thus, to understand whether or not the torsion grows, we need to understand whether or not the quantity $R$ can be very large- -- in particular, can it be exponentially large in the volume of $M$?
(Note, however, that even a complete understanding of $R$ would not settle the question of torsion growth; there remains the closely related and extremely difficult issue of small eigenvalues of the Laplacian). 

 In general, we see no reason that $R$ should be small. However,
{\em in the arithmetic case}, numerical experiments suggest that $H_1(M, \Z)_{\tors}$
is large whether or not the $H_1(M, \Q)$ is zero.    Moreover,  we saw in the combinatorial case (Lemma \ref{gaction})
that one may prove upper bounds for (combinatorial) regulators using the auxiliary structure of a group action; one may suspect that Hecke operators might play a similar role. Accordingly, we formulate:

\begin{conj} \label{regconj} 
Suppose that $M_k$ is a sequence of congruence arithmetic hyperbolic $3$-manifolds,
of injectivity radius approaching $\infty$. Then
$$ \lim_{k} \frac{ \log R(M_k) } {\vol(M_k)} = 0.$$ 
\end{conj}

To clarify the difficulty here, we point out that we do not know how to prove this even if, for example, we know that $\dim H_1(M_k, \C) = 1$ for every $k$. Rather, the difficulty is this:  if $\dim H_1(M, \Q) = 1$ it is usually rather easy to find an explicit element $\gamma \in \pi_1(M)$
whose image generates $H_1(\Q)$ {\em but it may not generate $H_1(M, \Z)$.} 
 (It is likely that this is related to the Gromov norm on $H_2$.)

As remarked, we  certainly do not conjecture this in the nonarithmetic case, and in fact are more inclined to think it false.  In a related vein, we suspect that the answer to the following question is YES, in contrast to the arithmetic case. 

\medskip 

{\em Question.  Do there exist   hyperbolic $3$-manifolds of arbitrarily large injectivity radius with torsion-free $H_1$?}

\subsection{Weight one modular forms, modulo $p$.}

Now suppose $G = \SL_2(\R)$. Curiously enough, even though $\rank(G) - \rank(K) = 0$
and there is no torsion in the homology of (torsion-free) lattices, 
there is a natural candidate for a phenomenon that ``mirrors'' the exponential growth of torsion. 
This was observed in the joint work of the second author with F. Calegari, cf. \cite{CV}. 
We explicate it only in the most concrete case:

Let $\pi: X_0(N) \rightarrow \Spec \Z[\frac{1}{6N}]$ be the modular curve of level $N$ --  regarded
as a $\Z[\frac{1}{6N}]$-scheme -- and let $\Omega = \Omega^1_{\mathscr{E}/X_0(N)}$
be the bundle of relative $1$-forms corresponding to the universal generalized elliptic curve;
let $\omega$ be its pull-back to $X_0(N)$ by means of the zero-section. 
Set $M_1(N) = \mathrm{R}^1\pi_* \omega$, which we identify (via its global sections) with
a finite rank module over $\Z[1/N]$.   
\medskip 

{\em Question.  Does the size $| M_1(N)_{\tors}|$ grow exponentially with $N$? }

\medskip

Suppose that $M_1(N) [p]$ is nonzero; the long exact sequence in cohomology
associated to the sequence of sheaves $\omega\stackrel{\times p}{\rightarrow} \omega \rightarrow 
\omega/p\omega$
shows that there is a  level $N$ weight $1$ modular form {\em modulo $p$} that fails to lift to characteristic zero.     
Parallel to analytic torsion: the size of $M_1(N)$ is related
to the determinant of a holomorphic Laplacian.

 \subsection{The $p$-part.}
 
 Throughout this paper, we have been concerned with the ``crude'' question
 of the size of torsion homology; we may also ask, more finely, about its group structure.  
 
 Notation as in Conjecture \ref{conjmain}; let $p$ be a prime. If $X$ is a finite abelian group, 
 we denote by $X_p$ the set of $x\in X$ so that $p^n x =0$ for some $n \geq 0$. 
  
  \medskip
{\em Question.  What can one say about limit $\lim \frac{\log  | H_j(\Gamma_N, M)_{p}  |}{[\Gamma: \Gamma_N]}$?}
\medskip
 
See  Silver and Williams \cite[Theorem 4.2]{SilverWilliams2} for cyclic covers. 
 The question may be of most interest when the $\Gamma_N$ are obtained by adding $p$-power level structure to $\Gamma_1$. In this case, one expects the asymptotic behavior of $H_j(\dots)_p$ to be related to the dimension of a suitable eigenvariety. We refer to the works of Calegari and Emerton \cite{CE,CE2} for further discussion of this.  These suggest that the torsion predicted by our conjecture \ref{conjmain} involves larger and larger sporadic primes.
  
In a different vein, we may ask:
\medskip 

{\em Question.
 What is the {\em distribution} of the isomorphism class of $H_i(\Gamma, M)_{p}$, when one varies $\Gamma$ through arithmetic subgroups? } 

\medskip 

This question is very vague, of course; one needs to be specific about the variation of $\Gamma$. 
In particular, it is desirable to ensure that the normalizer of each $\Gamma$ is as small
as possible, so that the picture is not clouded by extra automorphisms. 
For example, the normalizers of the standard subgroups $\Gamma_0(\mathfrak{p})$ of $\SL_2(\Z[i])$, 
where $\mathfrak{p}$ is a prime of $\Z[i]$, might be a suitable family. 

For example, let us suppose we have fixed a family $\{\Gamma_{\alpha}\}_{\alpha \in A}$ of 
arithmetic subgroups, and consider:
$$X_k := \mbox{fraction of $\alpha \in A$ for which $\dim H_i(M) \otimes \mathbb{F}_p =k$.}$$
If the distribution is governed by ``Cohen-Lenstra heuristics'' (see below), we expect
that $X_k$ might decay very rapidly as $k$ grows -- as $p^{-k^2}$. 
On the other hand, the existence of Hecke operators may cause a substantial deviation (this phenomena may be similar to the observed difference in eigenvalue spacings between arithmetic and nonarithmetic Fuchsian groups \cite{Sarnak}); one might perhaps expect that $X_k$ would decay rather 
as $p^{-k}$.  We do not know; for a certain model of (usually nonarithmetic) hyperbolic $3$-manifolds this question has been studied by Dunfield and Thurston \cite{DT}. The probability that
the homology group of a random $3$-manifold $H_1 (M , \mathbb{F}_p)$ is non-zero 
is of size roughly $p$. Again, because the series of $1/p$
over primes diverges, this suggests that the first homology of a random $3$-manifold is typically
finite, but is divisible by many primes, see \cite{Kowalski} for a quantified version of this.

{\em Remark.} Let $t \geq 0$. There is a unique probability distribution $\mu_t$ on isomorphism classes of 
finite abelian $p$-groups  -- the {\em Cohen-Lenstra distribution with parameter $t$},
cf. \cite{CL}  -- characterized in the following equivalent ways:
\begin{enumerate}
\item The distribution of the cokernel of a random map $\Z_p^{N+t} \rightarrow \Z_p^N$ ({\em random} according to the additive Haar measure on the space of such maps) approaches $\mu_t$, as $N \rightarrow \infty$. 
\item $\mu_t(A)$ is proportional to $|A|^{-t} |\Aut(A)|^{-1}$. 
\item Let $G$ be any finite abelian group; the expected number of homomorphisms from a $\mu_t$-random group into $G$ equals $|G|^{-t}$. 
\end{enumerate}

\bibliography{bibli}

\bibliographystyle{plain}

\end{document}